\documentclass[12pt,xcolor=dvipsnames]{amsart}
\usepackage[colorlinks=true]{hyperref}
\usepackage[utf8]{inputenc}
\usepackage{amsmath,amsthm,amssymb,url,mathdots,verbatim,psfrag,color}
 \usepackage[margin = 2.5cm]{geometry}
\usepackage{tikz}
\usepackage{ytableau}
\usepackage{color}
\usepackage{MnSymbol}
\usepackage{eurosym}
\usepackage{graphicx}
\usepackage{scalerel}
\usepackage{enumitem}
\usepackage[normalem]{ulem}
\newtheorem{theorem}{Theorem}

\newtheorem{prop}[theorem]{Proposition}

\newtheorem{lemma}[theorem]{Lemma}
\newtheorem{definition}[theorem]{Definition}

\theoremstyle{definition}

\numberwithin{equation}{section}
\numberwithin{theorem}{section}

\renewcommand{\u}[1]{\underline{#1}}

\newcommand{\x}{{\bf x}}

\DeclareMathOperator{\id}{id}

\DeclareMathOperator{\DPP}{{\mathbf{DPP}}}
\DeclareMathOperator{\ASM}{{\mathbf{ASM}}}

\DeclareMathOperator{\sgn}{sgn}

\newcommand{\nenwarrow}{\nwarrow \!\!\!\!\!\;\!\! \nearrow}
\usepackage{enumitem}
\newcommand{\si}[2]{\u{[#1,#2]}}

\newcommand{\e}{{\operatorname{E}}}
\newcommand{\W}{{\operatorname{W}}}
\newcommand{\wne}{\omega(\nearrow)}
\newcommand{\wnw}{\omega(\nwarrow)}
\newcommand{\wnenw}{\omega(\nenwarrow)}
\newcommand{\wnone}{\omega(\emptyset)}
\newcommand{\dec}{\operatorname{decor}}
\newcommand{\asym}{\mathbf{ASym}}

\newcommand{\nenwarrowij}[2]{{^{#1 \hspace{-1mm}} \nwarrow \!\!\!\!\!\;\!\! \nearrow^{#2}}}
\newcommand{\nenwarrowijs}[2]{{^{#1 \hspace{-0.5mm}} \nwarrow \!\!\!\!\!\;\!\! \nearrow^{#2}}}

\definecolor{dg}{rgb}{0.0, 0.2, 0.13}

\title[The ASM-DPP relation]{The relation between alternating sign matrices and descending plane partitions: $n+3$ pairs of equivalent statistics}

\author{Ilse Fischer}
\address{University of Vienna, Austria}
\urladdr{https://www.mat.univie.ac.at/~ifischer/}
\author{Florian Schreier-Aigner}
\address{University of Vienna, Austria}
\urladdr{https://homepage.univie.ac.at/florian.aigner/}
\thanks{Florian Schreier-Aigner acknowledges the financial support from the Austrian Science Foundation FWF, grant J 4387}

\begin{document}

\begin{abstract}
There is the same number of $n \times n$ alternating sign matrices (ASMs) as there is of descending plane partitions (DPPs) with parts no greater than $n$, but finding an explicit bijection is an open problem for about $40$ years now. So far, quadruples of statistics on ASMs and on DPPs that have the same joint distribution have been identified. We introduce extensions of ASMs and of DPPs along with $n+3$ statistics on each extension, and show that the two families of statistics have the same joint distribution. The ASM-DPP equinumerosity is obtained as an easy consequence by considering the $(-1)$-enumerations of these extended objects with respect to one pair of the $n+3$ pairs of statistics. One may speculate that the fact that these extensions might be necessary to have this significant increase in the number of statistics, as well as the involvement of signs when specializing to ASMs and DPPs may hint at the obstacles in finding an explicit bijection between ASMs and DPPs. One important tool for our proof is a multivariate generalization of the operator formula for the number of monotone triangles with prescribed bottom row that generalizes Schur functions.
\end{abstract}

\maketitle

\section{Introduction} 

An \emph{alternating sign matrix} (ASM) is a square matrix with entries in $\{0,\pm 1\}$ such that, in each row and each column,  $1$'s and $-1$'s alternate and sum to $1$. All $3 \times 3$ ASMs are given next.
$$
\tiny
\begin{pmatrix} 1 & 0 & 0 \\ 0 & 1 & 0 \\ 0 & 0 & 1 \end{pmatrix} \quad
\begin{pmatrix} 1 & 0 & 0 \\ 0 & 0 & 1 \\ 0 & 1 & 0 \end{pmatrix} \quad
\begin{pmatrix} 0 & 1 & 0 \\ 1 & 0 & 0 \\ 0 & 0 & 1 \end{pmatrix} \quad
\begin{pmatrix} 0 & 1 & 0 \\ 0 & 0 & 1 \\ 1 & 0 & 0 \end{pmatrix} \quad
\begin{pmatrix} 0 & 0 & 1 \\ 1 & 0 & 0 \\ 0 & 1 & 0 \end{pmatrix} \quad
\begin{pmatrix} 0 & 0 & 1 \\ 0 & 1 & 0 \\ 1 & 0 & 0 \end{pmatrix} \quad
\begin{pmatrix} 0 & 1 & 0 \\ 1 & -1 & 1 \\ 0 & 1 & 0 \end{pmatrix}
$$
ASMs had been introduced by Robbins and Rumsey in the 1980s \cite{lambda} in the course of generalizing the determinant to the so-called \emph{$\lambda$-determinant}, which turned out to be expressible as a sum over all ASMs of fixed size.
They conjectured that the number of $n \times n$ ASMs is 
\begin{equation}
\label{formula}  
\prod_{i=0}^{n-1} \frac{(3i+1)!}{(n+i)!},
\end{equation} 
which was proved after considerable effort by Zeilberger \cite{Zei96a} (in fact, he showed a more general result that includes an additional parameter). Soon after that Kuperberg used methods from statistical physics to give another, shorter proof \cite{Kup96} (of the special case). 

Two further observations caused considerable excitement in the combinatorics community: it turned out that \eqref{formula} appears also as the enumeration formula for two classes of plane partitions, namely for \emph{descending plane partitions} (DPPs) \cite{MilRobRum83} and for \emph{totally symmetric self-complementary plane partitions} (TSSCPPs) \cite{MilRobRum86}. Since then, there has been much effort to construct explicit bijections that prove these facts, but until recently without much progress. A beautiful account on the early story of this line of research is described in \cite{Bre99}. 

Very recently, a bijective proof of the fact that $n \times n$ ASMs are counted by \eqref{formula} has been constructed as well as a bijective proof of an identity that implies the equinumerosity of ASMs and DPPs, see \cite{partI,partII,cube}. More precisely, if $\ASM_n$ denotes the set of $n \times n$ ASMs and $\DPP_n$ denotes the set of DPPs with parts no greater than $n$, a bijection between $\ASM_n \times \DPP_{n-1}$ and $\ASM_{n-1} \times \DPP_n$ has been constructed (among other things).
These constructions are quite involved, and use signed sets and a generalization of the involution principle by Garsia and Milne  \cite{garsiamilne3,garsiamilne2}. It seems fair to say that a perfect understanding of these relations despite the many efforts is still missing. To complete the picture, it should also be mentioned that recently a new class of objects that are counted by \eqref{formula} has been discovered \cite{extreme}, namely \emph{alternating sign triangles} (ASTs), which are as indicated by the name definitely more on the ASM side. 
\bigskip

We now define DPPs and a variant needed here.
A \emph{strict partition} is a sequence $\lambda=(\lambda_1,\ldots,\lambda_n)$ of positive integers with $\lambda_1 > \ldots > \lambda_n$, and the \emph{shifted Young diagram} of shape $\lambda$ has $\lambda_i$ cells in row $i$ and each row is indented by one cell to the right with respect to the previous row. The shifted Young diagram of the strict partition $(5,4,2,1)$ is displayed next.
$$
\tiny
\ydiagram{0+5,1+4,2+2,3+1}
$$
A \emph{column strict shifted plane partition}  (CSSPP) is a filling of a shifted Young diagram with positive integers such that rows decrease weakly and columns decrease strictly. An example is given next. 
$$
\tiny
\begin{ytableau}
7 & 6 & 6 & 5 & 5 \\
\none & 5 & 5 & 4 &4  \\
\none & \none & 3 & 3 \\
\none & \none & \none & 2 
\end{ytableau}
$$
A CSSPP is said to be a \emph{descending plane partition} (DPP) if the first part of each row is greater than the length of its row and less than or equal to the length of the previous row. The example is obviously a DPP. The DPPs with all entries less than or equal to $3$ are given next.
$$\emptyset \quad  \tiny \begin{ytableau} 2 \end{ytableau} \quad  \begin{ytableau} 3 \end{ytableau} \quad \begin{ytableau} 3 &  1 \end{ytableau} \quad
\begin{ytableau} 3 & 2 \end{ytableau} \quad \begin{ytableau} 3 & 3 \end{ytableau} \quad \begin{ytableau} 3 & 3 \\ \none & 2 \end{ytableau}$$
It is no coincidence that the number of these objects is the same as the number of $3 \times 3$ ASMs as it is known that  the number of DPPs with parts no greater than $n$ is also given by \eqref{formula} for general $n$. This was shown by Andrews \cite{And79} prior to the introduction of ASMs.

We introduce other subclasses of CSSPPs. Let $k$ be an integer, then a CSSPP is said to be of class $k$ if the first part of each row exceeds the length of its row by precisely $k$. It can be seen that DPPs are equivalent to CSSPPs of class $2$. In order to transform a DPP into a CSSPP of class $2$ add $1$ to each part and add possibly parts equal to $1$ at the end of each row to guarantee that all conditions are satisfied in the end.\footnote{More precisely, letting $\pi_{i,i}$ be the first entry in row $i$ of the DPP and $\lambda_i$ the length of row $i$, we have $\lambda_i+1 \leq \pi_{i,i} \leq \lambda_{i-1}$ (setting $\lambda_{0}=\infty$). We add $1$ to each entry and append $\pi_{i,i}-\lambda_i-1$ $1$'s at the end of row $i$ for each $i$ so that the new length of row $i$ is just $\pi_{i,i}-1$, which is exceeded by 
 the new first entry in row $i$ by $2$. In order to see that the columns are strictly decreasing, it suffices to show that the top neighbors of the added $1$'s at the end of the rows are greater than $1$. Indeed, for $i\geq 2$, the new entries equal to $1$ in row $i$ appear until column $\pi_{i,i}-1+(i-1)$, while in row $i-1$, the entries are at least $2$ until column $\lambda_{i-1}+(i-2)$, and the assertion follows from the assumption $\pi_{i,i} \le \lambda_{i-1}$.}
In our example, this transformation results in the following CSSPP of class $2$.
$$
\tiny
\begin{ytableau}
8 & 7 & 7 & 6 & 6 & 1\\
\none & 6 & 6 & 5 &5  \\
\none & \none & 4 & 4 \\
\none & \none & \none & 3 
\end{ytableau}
$$ 
It should also be mentioned that CSSPPs of class $k$ are known to be in ``easy'' bijective correspondence with cyclically symmetric lozenge tilings of a hexagon with a central triangular hole of size $k$ \cite{Kra06}.

The lack of an explicit bijection between $n \times n$ ASMs and DPPs with parts no greater than $n$ is also astonishing in the view of the fact that Mills, Robbins and Rumsey had discovered triples of statistics for both objects for which they conjectured that they have the same joint distribution \cite{MilRobRum83}. This conjecture was proved in \cite{BehDifZin12}, and in \cite{BehDifZin13} another statistic was 
 added, which was the best possible result in this direction so far. 

\medskip

A main purpose of this paper is to shed light on the relation between ASMs and DPPs, in particular it may offer an  understanding why it might be hard to find the longed for explicit bijection in the current setting, and it provides a new, extended setting in which it could be (arguably) easier to find an explicit bijection. More concretely, we define extended (multivariate) weighted objects, both on the ASM side and on the DPP side, and show that the generating functions of these extended objects coincide. The equinumerosity of ASMs and DPPs is then a simple consequence of this result that is obtained by specializing the weights in a way such that objects are assigned weights $\pm 1$. Simple sign reversing involutions on both sides ``cancel'' certain subsets of the extended objects, and the remaining two sets are in easy bijective correspondence with the classical ASMs and DPPs, respectively, see the discussion after 
\eqref{specialization} and Proposition~\ref{prop: SBCSPP and DPP}.
The use of signed sets in the current work is in accordance with the recent (complicated) bijective proofs mentioned above. It is an open problem whether this extension to signed sets is needed to allow combinatorial proofs in this area or whether there is  (after so many years) a still to be discovered other approach that makes a sign-free combinatorial reasoning possible.

The major advantage of the extended objects is that we were able to make significant progress concerning the number 
of pairs of equivalent statistics we have identified on the two families of objects. More precisely, in this paper we are able to perform the step from a constant number of statistics (the four from  \cite{BehDifZin13}) to a linear number of statistics in the order $n$, namely $n+3$.
In a sense, a further step to the (right) quadratic number of statistics would resolve the problem of establishing a bijection as the two families of objects are (in a sense) assembled of  a quadratic number  of integers at most.\footnote{For instance, a monotone triangle of order $n$ is comprised of $1+2+\ldots+n=\frac{(n+1)n}{2}$ integers, which we could also view as statistics (note that the top entry of a monotone triangle is a classical statistics that as been studied, in particular an equivalent statistics on DPPs has been identified, see \cite{MilRobRum83, BehDifZin12}). 
If we were to find a set of $\frac{(n+1)n}{2}$ statistics on the DPP-side that has the same joint distribution, then the bijection would have been constructed.
{In the extended setting that we propose in this paper, ASMs are replaced by arrowed monotone triangles of order $n$, and they are comprised of $\frac{(n+1)n}{2}$ integers and $\frac{(n+1)n}{2}$ decorations and thus determined by $(n+1)n$ statistics.}}

\section{Main definitions and main results} 
\label{sectmain}
Recall that a \emph{monotone triangle} is a triangular array of integers of the following form 
$$
\begin{array}{ccccccccccccccccc}
  &   &   &   &   &   &   &   & m_{1,1} &   &   &   &   &   &   &   & \\
  &   &   &   &   &   &   & m_{2,1} &   & m_{2,2} &   &   &   &   &   &   & \\
  &   &   &   &   &   & \dots &   & \dots &   & \dots &   &   &   &   &   & \\
  &   &   &   &   & m_{n-2,1} &   & \dots &   & \dots &   & m_{n-2,n-2} &   &   &   &   & \\
  &   &   &   & m_{n-1,1} &   & m_{n-1,2} &  &   \dots &   & \dots   &  & m_{n-1,n-1}  &   &   &   & \\
  &   &   & m_{n,1} &   & m_{n,2} &   & m_{n,3} &   & \dots &   & \dots &   & m_{n,n} &   &   &
\end{array}
$$
with weak increase along 
northeast- and 
southeast-diagonals, and strict increase along rows, i.e., $m_{i+1,j} \le m_{i,j} \le m_{i+1,j+1}$ and 
$m_{i,j} < m_{i,j+1}$. Monotone triangles with bottom row $(1,2,\ldots,n)$ are in easy bijective correspondence with $n \times n$ ASMs, as explained briefly at the beginning of Section~\ref{damt}.
 We define certain 
 decorated monotone triangles.

\begin{definition} 
\label{ramt}
An \emph{arrowed monotone triangle (AMT)} is a monotone triangle where each entry $e$ carries a decoration from $\{\nwarrow, \nearrow, \nenwarrow \}$ and the following two conditions are satisfied: 
\begin{itemize} 
\item If $e$ has a 
northwest-neighbor and is equal to it, then $e$ must \emph{not} carry $\nwarrow,\nenwarrow$ 
(and therefore carries $\nearrow$).
\item If $e$ has a 
northeast-neighbor and is equal to it, then $e$ must \emph{not} carry $\nearrow,\nenwarrow$
(and therefore carries $\nwarrow$). 
\end{itemize}
In summary, the arrows simply indicate that the entries are different from the next entries in the specified directions.
\end{definition}

When assigning $\nenwarrow$ to an entry $a$ in our examples, we write ${^{\nwarrow}{a}^{\nearrow}}$. Such an example is given next.
$$
\begin{array}{ccccccccccccccccc}
  &   &   &   &   &   &   &   & 4^{\nearrow} &   &   &   &   &   &   &   & \\
  &   &   &   &   &   &   & ^{\nwarrow}2^{\nearrow} &   & ^{\nwarrow}5 &   &   &   &   &   &   & \\
  &   &   &   &   &   & ^{\nwarrow}2 &   & 3^{\nearrow} &   & {5}^{\nearrow} &   &   &   &   &   & \\
  &   &   &   &   & 1^{\nearrow} &   & ^{\nwarrow}3 &   & ^{\nwarrow}{4}^{\nearrow} &   & 6^{\nearrow} &   &   &   &   & \\
  &   &   &   & ^{\nwarrow}1 &   & 2^{\nearrow} &  &   3^{\nearrow} &   & ^{\nwarrow}5   &  & 6^{\nearrow}  &   &   &   & \\
  &   &   & ^{\nwarrow}1 &   & ^{\nwarrow}2 &   & ^{\nwarrow}3 &   & ^{\nwarrow}4^{\nearrow}  &   & 5^{\nearrow} &   & 6^{\nearrow} &   &   &
\end{array}
$$
We associate the following weight to an arrowed monotone triangle with $n$ rows, 
\begin{equation}
u^{\# \nearrow} v^{\# \nwarrow} w^{\# \nenwarrow} \prod_{i=1}^n X_i^{\text{(sum of entries in row $i$)}-\text{(sum of entries in row $i-1$)}+(\# \nearrow \text{in row $i$}) - (\# \nwarrow \text{in row $i$})}
\end{equation} 
where the sum of entries in row $0$ is defined to be $0$. As we will show in Proposition~\ref{downup}, the sum of the weights of all arrowed monotone triangles with a fixed underlying ordinary monotone triangle (i.e.,\ without arrows) is just the weight of that monotone triangle multiplied by $\prod_{i=1}^n (u X_i + v X_i^{-1} + w)$.

The exponents of the $n+3$ variables are the statistics from the title for the ASM-side. The weight of the arrowed monotone triangle given above is 
$$
u^{10} v^{8} w^3  X_1^{5} X_2^{2} X_3^{4} X_4^{5} X_5^{4} X_6^{3}. 
$$
Our first main result concerns the generating function of arrowed monotone triangles with prescribed bottom row. It is a multivariate generalization of the operator formula from \cite{Fis06} (with extensions given in \cite{Fis10}). It requires a (straightforward) extension of the Schur polynomial $s_{(k_n,k_{n-1},\ldots,k_1)}(X_1,\ldots,X_n)$ to integer sequences $k_1,\ldots,k_n$ that are not necessarily weakly increasing. This extension is provided in Section~\ref{ASM}.  

\begin{theorem} 
\label{main1}
The generating function of arrowed monotone triangles with bottom row $k_1 < \ldots < k_n$ is 
\begin{equation} 
\label{opX}
\prod_{i=1}^{n} (u X_i + v X_i^{-1} + w) \\ 
\prod_{1 \le p < q \le n} \left( u  \e_{k_p} + v \e_{k_q}^{-1} + w \e_{k_p} \e_{k_q}^{-1}  \right) \\
s_{(k_n,k_{n-1},\ldots,k_1)}(X_1,\ldots,X_n), 
\end{equation} 
where $\e_x$ denotes the shift operator, defined as $\e_x p(x) = p(x+1)$. 
\end{theorem} 
The formula has to be interpreted as follows: First treat the $k_i$'s in $s_{(k_n,k_{n-1},\ldots,k_1)}(X_1,\ldots,X_n)$ as indeterminates, and  apply $\prod_{1 \le p < q \le n} \left( u \e_{k_p} + v \e_{k_q}^{-1} + w \e_{k_p} \e_{k_q}^{-1} \right)$. Note that applying the double product means that we first expand it and sum over all monomials in the expansion, which we then apply to the Schur polynomial. Only then the $k_i$'s can be specialized to the actual values. We illustrate this with the help of the case $n=2$.  Namely, the generating function of monotone triangles with bottom row 
$k_1,k_2$ is 
\begin{multline*}
(u X_1 + v X_1^{-1} + w)(u X_2 + v X_2^{-1} + w) \left( u  \e_{k_1} + v \e_{k_2}^{-1} + w \e_{k_1} \e_{k_2}^{-1}  \right) s_{(k_2,k_{1})}(X_1,X_2) \\
= (u X_1 + v X_1^{-1} + w)(u X_2 + v X_2^{-1} + w) \\
\times 
\left(u s_{(k_2,k_1+1)}(X_1,X_2) + v s_{(k_2-1,k_1)}(X_1,X_2) + w s_{(k_2-1,k_{1}+1)}(X_1,X_2) \right).
\end{multline*}

We provide a variation of this result in Theorem~\ref{robbins0}, but also significant extensions in two different directions: in Theorem~\ref{robbins}, we deal with arbitrary integer sequences $k_1,\ldots,k_n$ (with the proofs given in  Section~\ref{constantterm}), while, in Theorem~\ref{robbinsgen}, we consider (in addition) much more general decorations.

It turns out that the generating function of arrowed monotone triangles with bottom row $k_1,\ldots,k_n$ gives the number of (plain) monotone triangles with bottom row $k_1,\ldots,k_n$ when setting in the former
\begin{equation}
\label{specialization}
u=v=1, w=-1, (X_1,\ldots,X_n)=(1,\ldots,1).
\end{equation} 
This follows from Theorem~\ref{robbins0} and its connection to Theorem~\ref{main1}, as discussed  in Section~\ref{implies}, however, we present an elementary argument here.
 An entry in a monotone triangle that is different from its northwest-neighbour and different from its northeast-neighbour is said to be free (if one or both of the neighbours do not exist, then they are also treated as being different from the entry). Such entries may be equipped with any of the three decorations to obtain an arrowed monotone triangle, for all other entries the decoration is prescribed as either $\nwarrow$ or $\nearrow$. The sum of all weights of arrowed monotone triangles that are obtained from a given (plain) monotone triangle by eligible decorations is, after setting $(X_1,\ldots,X_n)=(1,\ldots,1)$, equal to $(u+v+w)^{f} u^{a} v^{b}$, where $f$ is  the number of free entries, $a$ is the number of entries equal to their northwest-neighbours and $b$ is the number of entries equal to their northeast-neighbours. When setting also $u=v=1$ and $w=-1$, this sum simplifies to $1$.
\bigskip

Our second main result concerns the special case $(k_1,\ldots,k_n)=(1,2,\ldots,n)$ (which concerns the case of ASMs). More precisely, we define certain weighted plane partitions, whose generating function coincides with the 
one for arrowed monotone triangles in this special case. Excitingly, it can be seen combinatorially that the generating function of these plane partitions with parts no greater than $n$ is equal to the number of DPPs with parts no greater than $n$ when specializing as in \eqref{specialization}, see Proposition~\ref{prop: SBCSPP and DPP}.

For the definition of these plane partitions, almost self-conjugate shapes in the following sense are crucial.

\begin{definition} 
\label{balanced}
Let $\lambda=(a_1,\ldots,a_l|b_1,\ldots,b_l)$ be a partition in Frobenius notation, i.e., $a_i$ is the number of cells right of the diagonal cell $(i,i)$ in the same row, while $b_i$ 
is the number of cells below $(i,i)$ in the same column. We say that $\lambda$ is \emph{near-balanced} if, for all $i$, either $a_i=b_i$ or $a_i=b_i+1$. 
The weight is 
\begin{equation} 
\W(\lambda) = w^{l+\sum_{i=1}^l (b_i - a_i)}. 
\end{equation} 
Phrased differently, the exponent of $w$ is the number of diagonal cells that have a balanced hook.
\end{definition}

We list all near-balanced shapes with at most $3$ rows. 
\begin{multline*} 
\emptyset, \tiny   \begin{ytableau}  \empty \\ \end{ytableau}, 
\tiny \begin{ytableau} \empty &  \empty \end{ytableau}, 
\begin{ytableau} \empty &  \empty \\ \empty \\     \end{ytableau},
\begin{ytableau} \empty &  \empty & \empty  \\ \empty  \\   \end{ytableau},
\begin{ytableau} \empty &  \empty & \empty  \\ \empty   \\ \empty  \\ \end{ytableau},
\begin{ytableau} \empty &  \empty & \empty & \empty \\ \empty   \\ \empty  \\ \end{ytableau},
\begin{ytableau} \empty &  \empty \\ \empty & \empty \\     \end{ytableau},
\begin{ytableau} \empty &  \empty & \empty  \\ \empty & \empty  \\   \end{ytableau}, 
\begin{ytableau} \empty &  \empty & \empty  \\ \empty & \empty  \\ \empty  \\ \end{ytableau},
\begin{ytableau} \empty &  \empty & \empty & \empty \\ \empty & \empty   \\ \empty  \\ \end{ytableau}, 
\begin{ytableau} \empty &  \empty & \empty \\ \empty & \empty & \empty \\  \end{ytableau}, \\ \tiny
\begin{ytableau} \empty &  \empty & \empty \\ \empty & \empty & \empty \\ \empty \\  \end{ytableau},
\begin{ytableau} \empty &  \empty & \empty  & \empty \\ \empty & \empty & \empty \\ \empty \\  \end{ytableau},
\begin{ytableau} \empty &  \empty & \empty \\ \empty & \empty & \empty \\ \empty & \empty  \\  \end{ytableau},
\begin{ytableau} \empty &  \empty & \empty & \empty \\ \empty & \empty & \empty \\ \empty & \empty  \\ \end{ytableau},
\begin{ytableau} \empty &  \empty & \empty & \empty \\ \empty & \empty & \empty & \empty  \\ \empty & \empty  \\ \end{ytableau},
\begin{ytableau} \empty &  \empty & \empty \\ \empty & \empty & \empty \\ \empty & \empty & \empty  \\  \end{ytableau},
\begin{ytableau} \empty &  \empty & \empty & \empty \\ \empty & \empty & \empty \\ \empty & \empty & \empty  \\ \end{ytableau},
\begin{ytableau} \empty &  \empty & \empty & \empty \\ \empty & \empty & \empty & \empty  \\ \empty & \empty & \empty   \\ \end{ytableau}, 
\begin{ytableau} \empty &  \empty & \empty & \empty  \\ 
\empty & \empty & \empty & \empty \\ 
\empty & \empty & \empty & \empty \end{ytableau}
\end{multline*} 
Their weights are 
$$1,w, 1,w, 1, w, 1, w^2, w, w^2, w, 1, w, 1, w^2, w, 1, w^3, w^2, w,1.$$

Now we are in the position to define the new class of plane partitions.

\begin{definition} 
\label{PP}
A set-valued near-balanced column strict plane partition (SBCSPP) $D$ of shape $\lambda$ and order $n$ is a filling of a near-balanced  partition $\lambda$ with non-empty subsets of  $\{1,2,\ldots,n\}$ such that strictly above the diagonal the subsets are singletons, and 
\begin{enumerate} 
\item  rows decrease weakly in the sense that the maxima of the sets form a weakly decreasing sequence if read from left to right, and 
\item columns decrease strictly in the sense that for two adjacent cells in a column, all elements in the top cell are strictly greater than all elements in the bottom cell.
\end{enumerate}
The weight of $D$ is as follows
\begin{multline} 
\W(D) = \W(\lambda) \cdot u^{\# \text{of cells strictly above the main diagonal}}  
\cdot v^{\binom{n+1}{2} -\# \text{of entries on and below the main diagonal}} \\
\cdot w^{\# \text{of entries} - \# \text{of cells}} \cdot \prod_{i=1}^n X_i^{\# \text{ of $i$'s in D}} 
\end{multline} 
\end{definition}
Again, the exponents of the $n+3$ variables in the weight are the statistics from the title for the DPP-side.
An example of an SBCSPP is displayed next.
$$
\begin{ytableau}
 8 & 8 & 8 & 7 & 7 & 6 & 4 & 1 \\
 7 &  7  & 7 & 6 & 5 & 5  \\
  6 &   6 &  5  & 4 & 4 & 4  \\
 5 &  4 &  3 &  3,2  & 3 & 2 \\
 3 &  2 &  2,1 &  1 \\
 2 & 1 \\
 1
\end{ytableau}
$$
Letting $n=9$, its weight is 
$$
u^{16} v^{26} w^{3} X_1^{5} X_2^{5} X_3^{4} X_4^5 X_5^4 X_6^4 X_7^5 X_8^3.
$$
Before stating our second main result, which connects arrowed monotone triangles with bottom row 
$1,2,\ldots,n$ and SBCSPPs of order $n$, we elaborate on the connection between SBCSPPs and DPPs. 
(More details about the DPP-side are provided in Section~\ref{dpp}.)

\begin{prop}
\label{prop: SBCSPP and DPP}
The weighted sum over all SBCSPPs of order $n$ when setting $u=v=1, w=-1$ and $(X_1,\ldots,X_n)=(1,\ldots,1)$ is equal to the number of DPPs with parts no greater than $n$.
\end{prop}
\begin{proof}
In the above specialisation, all SBCSPPs have weight $\pm 1$. In the following, we define a sign-reversing involution on a subset of the SBCSPPs of order $n$ and then see that the remaining SBCSPPs are in easy bijection with DPPs with parts no greater than $n$. 

 We say that an SBCSPP is \emph{principal} if the sets in all cells are singletons. We associate with 
each SBCSPP a principal SBCSPP by just keeping the maximum of each set in the cells. Clearly, this always produces a principal SBCSPP and also each principal SBCSPP is assigned to itself. We define two sign-reversing involutions.
\begin{enumerate}[label=(\emph{\alph*})]
\item For a given principal SBCSPP that has at least one other SBCSPP assigned to it, 
we define a sign-reversing involution on the set of all SBCSPPs which have this principal SBCSPP assigned to them. Consider the leftmost column with a cell that can contain a set of size at least $2$, and if there is more than one such cell in this column, then consider the bottommost of these cells. Then this cell has a unique minimal integer $i$ that is allowed to be in this set. In case $i$ is in this cell, we remove it; otherwise we add it. Clearly this yields a sign-reversing involution.
\item Thus it suffices to consider principal SBCSPPs that have only one SBCSPP assigned to them. They are characterized as follows: for each diagonal entry $d$, the entries below in the same column are $d-1,d-2,\ldots,1$ (in particular, there are $d-1$ cells below that cell). 
Next, we define a sign-reversing involution on the subset of these SBCSPPs for which at least one of the following is satisfied: the SBCSPP contains a $1$ strictly above the diagonal or $a_i \not= b_i+1$ for an $i$ where $(a_1,\ldots,a_l|b_1,\ldots,b_l)$ is the shape of the SBCSPP in Frobenius notation.

 If $a_i= b_i+1$ for all $i$, choose the topmost row that contains a $1$, and remove the rightmost cell (which then has to contain a $1$). Otherwise choose the minimal $i$ with 
$a_i \not= b_i+1$. If the SBCSPP has a $1$ above row $i$, do as we did in the previous case. Otherwise add a $1$ at the end of row $i$.
\end{enumerate}

The SBCSPPs that do not cancel under the two sign-reversing involutions can be described as follows: (1) all cells contain a single element, (2) their shapes $(a_1,\ldots,a_l|b_1,\ldots,b_l)$ satisfy $a_i = b_i+1$ for all $i$, (3) weakly below the diagonal entries, we have consecutive integers ending with $1$ down columns, and (4) there are no $1$'s above the diagonal.  All such SBCSPPs have weight $1$ in our specialization and are referred to as the DPP-SBCPPs in the following.

Next we remove all cells strictly below the main diagonal (no information is lost), and we obtain a column strict \emph{shifted} plane partition, i.e., the shape is not left-justified, but each row is intended by one cell to the right with respect to the previous row. From $a_i = b_i+1$, it follows that the first part of each row is one less than the length. Since there is no $1$ in the plane partition, we may subtract $1$ from each entry and obtain a column strict shifted plane partition such that the first part of each row is two less than the length of its row. Such CSSPPs with parts no greater than $n-1$ are in easy bijective correspondence with CSSPPs of class $2$ with parts no greater than $n+1$, by conjugating the partition in each row. The latter are in easy bijective correspondence with DPPs with parts no greater than $n$, as shown in the introduction.
\end{proof}

Our second main result is the following.

\begin{theorem} 
\label{main2}
The generating function of arrowed monotone triangles with bottom row $1,2,\ldots,n$ is equal to the generating function of set-valued near-balanced column strict plane partitions with parts in $\{1,2,\ldots,n\}$.
\end{theorem} 
 
We illustrate the theorem with the help of the case $n=2$ here; its proof is presented in Section~\ref{proofmain2}. This case is in fact small enough such that many weights appear precisely once, except for $u v w X_1 X_2^2$ and $u v w X_1^2 X_2$ which appear twice, and, therefore, there are exactly four weight-preserving bijections between arrowed monotone triangles with bottom row $1,2$ and SBCSPPs with parts in ${1,2}$.

\begin{center}
\begin{tabular}{ccc}
\begin{tabular}{|c|c|c|}
\hline
AMT & W & SBCSPP \\ \hline
\begin{tikzpicture}
	\node at (0,0) {$1$};
	\node at (1.2,0) {$2$};
	\node at (.6,.5) {$1$};
	\node at (-.3,.15) {$\nwarrow$};
	\node at (.9,.15) {$\nwarrow$};
	\node at (.3,.65) {$\nwarrow$};
	\end{tikzpicture}
	& \raisebox{4mm}{$v^3$}
	& \raisebox{4mm}{$\emptyset$} \\ \hline

\begin{tikzpicture}
	\node at (0,0) {$1$};
	\node at (1.2,0) {$2$};
	\node at (.6,.5) {$1$};
	\node at (-.3,.15) {$\nwarrow$};
	\node at (.9,.15) {$\nwarrow$};
	\node at (.3,.65) {$\nwarrow$};
	\node at (.9,.65) {$\nearrow$};
	\end{tikzpicture}
	& \raisebox{4mm}{$v^2wX_1$}
	& \raisebox{5mm}{\footnotesize \begin{ytableau}  1   \end{ytableau}} \\ \hline

\begin{tikzpicture}
	\node at (0,0) {$1$};
	\node at (1.2,0) {$2$};
	\node at (.6,.5) {$1$};
	\node at (-.3,.15) {$\nwarrow$};
	\node at (.9,.15) {$\nwarrow$};
	\node at (.9,.65) {$\nearrow$};
	\end{tikzpicture}
	& \raisebox{4mm}{$uv^2X_1^2$}
	& \raisebox{5mm}{\footnotesize \begin{ytableau} 1&1   \\ \end{ytableau}} \\ \hline

\begin{tikzpicture}
	\node at (0,0) {$1$};
	\node at (1.2,0) {$2$};
	\node at (.6,.5) {$1$};
	\node at (-.3,.15) {$\nwarrow$};
	\node at (.9,.15) {$\nwarrow$};
	\node at (1.5,.15) {$\nearrow$};
	\node at (.3,.65) {$\nwarrow$};
	\end{tikzpicture}
	& \raisebox{4mm}{$v^2wX_2$}
	& \raisebox{5mm}{\footnotesize \begin{ytableau}  2  \\ \end{ytableau}} \\ \hline

\begin{tikzpicture}
	\node at (0,0) {$1$};
	\node at (1.2,0) {$2$};
	\node at (.6,.5) {$2$};
	\node at (-.3,.15) {$\nwarrow$};
	\node at (1.5,.15) {$\nearrow$};
	\node at (.3,.65) {$\nwarrow$};
	\end{tikzpicture}
	& \raisebox{4mm}{$u v^2 X_1X_2$}
	& \raisebox{5mm}{\footnotesize \begin{ytableau}  2&1    \\ \end{ytableau}} \\ \hline

\begin{tikzpicture}
	\node at (0,0) {$1$};
	\node at (1.2,0) {$2$};
	\node at (.6,.5) {$1$};
	\node at (-.3,.15) {$\nwarrow$};
	\node at (.9,.15) {$\nwarrow$};
	\node at (1.5,.15) {$\nearrow$};
	\node at (.3,.65) {$\nwarrow$};
	\node at (.9,.65) {$\nearrow$};
	\end{tikzpicture}
	& \raisebox{4mm}{$vw^2X_1X_2$}
	& \raisebox{5mm}{\footnotesize \begin{ytableau}  2,1  \\ \end{ytableau}} \\ \hline

\begin{tikzpicture}
	\node at (0,0) {$1$};
	\node at (1.2,0) {$2$};
	\node at (.6,.5) {$1$};
	\node at (-.3,.15) {$\nwarrow$};
	\node at (1.5,.15) {$\nearrow$};
	\node at (.3,.65) {$\nwarrow$};
	\end{tikzpicture}
	& \raisebox{4mm}{$uv^2X_2^2$}
	& \raisebox{5mm}{\footnotesize \begin{ytableau}  2& 2  \\ \end{ytableau}} \\ \hline

\begin{tikzpicture}
	\node at (0,0) {$1$};
	\node at (1.2,0) {$2$};
	\node at (.6,.5) {$2$};
	\node at (-.3,.15) {$\nwarrow$};
	\node at (1.5,.15) {$\nearrow$};
	\node at (.3,.65) {$\nwarrow$};
	\node at (.9,.65) {$\nearrow$};
	
	\begin{scope}[yshift=-1.4cm]
	\node at (0,0) {$1$};
	\node at (1.2,0) {$2$};
	\node at (.6,.5) {$1$};
	\node at (-.3,.15) {$\nwarrow$};
	\node at (.9,.15) {$\nwarrow$};
	\node at (1.5,.15) {$\nearrow$};
	\node at (.9,.65) {$\nearrow$};
	\end{scope}
	\end{tikzpicture}
	& \raisebox{12mm}{\color{red} $u v w  X_1^2 X_2$}
	& \begin{tikzpicture}
	\node at (0,0) {\footnotesize \begin{ytableau}  2 & 1 \\ 1   \\ \end{ytableau}};
	\node at (0,-1.3) {\footnotesize \begin{ytableau}  2,1 & 1   \\ \end{ytableau}};
	\end{tikzpicture} \\ \hline
\end{tabular}
&
\begin{tabular}{|c|c|c|}
\hline
AMT & W & SBCSPP \\ \hline
\begin{tikzpicture}
	\node at (0,0) {$1$};
	\node at (1.2,0) {$2$};
	\node at (.6,.5) {$2$};
	\node at (-.3,.15) {$\nwarrow$};
	\node at (.3,.15) {$\nearrow$};
	\node at (1.5,.15) {$\nearrow$};
	\node at (.3,.65) {$\nwarrow$};
	\begin{scope}[yshift=-1.4cm]
	\node at (0,0) {$1$};
	\node at (1.2,0) {$2$};
	\node at (.6,.5) {$1$};
	\node at (-.3,.15) {$\nwarrow$};
	\node at (1.5,.15) {$\nearrow$};
	\node at (.3,.65) {$\nwarrow$};
	\node at (.9,.65) {$\nearrow$};
	\end{scope}	
	\end{tikzpicture}
	& \raisebox{12mm}{\color{blue} $u v w X_1 X_2^2$}
	&  \begin{tikzpicture}
	\node at (0,0) {\footnotesize \begin{ytableau}  2 & 2 \\ 1   \\ \end{ytableau}};
	\node at (0,-1.3) {\footnotesize \begin{ytableau}  2,1 &2  \\ \end{ytableau}};
	\end{tikzpicture} \\ \hline

\begin{tikzpicture}
	\node at (0,0) {$1$};
	\node at (1.2,0) {$2$};
	\node at (.6,.5) {$2$};
	\node at (-.3,.15) {$\nwarrow$};
	\node at (1.5,.15) {$\nearrow$};
	\node at (.9,.65) {$\nearrow$};
	\end{tikzpicture}
	& \raisebox{4mm}{$u^2vX_1^3X_2$}
	& \raisebox{5mm}{\footnotesize \begin{ytableau}  2 & 1 & 1 \\ 1   \\ \end{ytableau}} \\ \hline

\begin{tikzpicture}
	\node at (0,0) {$1$};
	\node at (1.2,0) {$2$};
	\node at (.6,.5) {$2$};
	\node at (-.3,.15) {$\nwarrow$};
	\node at (.3,.15) {$\nearrow$};
	\node at (1.5,.15) {$\nearrow$};
	\node at (.3,.65) {$\nwarrow$};
	\node at (.9,.65) {$\nearrow$};
	\end{tikzpicture}
	& \raisebox{4mm}{$uw^2X_1^2X_2^2$}
	& \raisebox{5mm}{\footnotesize \begin{ytableau}  2 & 2 \\ 1 & 1   \\ \end{ytableau}} \\ \hline

\begin{tikzpicture}
	\node at (0,0) {$1$};
	\node at (1.2,0) {$2$};
	\node at (.6,.5) {$1$};
	\node at (-.3,.15) {$\nwarrow$};
	\node at (1.5,.15) {$\nearrow$};
	\node at (.9,.65) {$\nearrow$};
	\end{tikzpicture}
	& \raisebox{4mm}{$u^2vX_1^2X_2^2$}
	& \raisebox{5mm}{\footnotesize \begin{ytableau}  2 & 2 & 1 \\ 1   \\ \end{ytableau}} \\ \hline

\begin{tikzpicture}
	\node at (0,0) {$1$};
	\node at (1.2,0) {$2$};
	\node at (.6,.5) {$2$};
	\node at (-.3,.15) {$\nwarrow$};
	\node at (.3,.15) {$\nearrow$};
	\node at (1.5,.15) {$\nearrow$};
	\node at (.9,.65) {$\nearrow$};
	\end{tikzpicture}
	& \raisebox{4mm}{$u^2wX_1^3X_2^2$}
	& \raisebox{5mm}{\footnotesize \begin{ytableau}  2 & 2 & 1 \\ 1 &1   \\ \end{ytableau}} \\ \hline

\begin{tikzpicture}
	\node at (0,0) {$1$};
	\node at (1.2,0) {$2$};
	\node at (.6,.5) {$2$};
	\node at (.3,.15) {$\nearrow$};
	\node at (1.5,.15) {$\nearrow$};
	\node at (.3,.65) {$\nwarrow$};

	\end{tikzpicture}
	& \raisebox{4mm}{$u^2vX_1X_2^3$}
	& \raisebox{5mm}{\footnotesize \begin{ytableau}  2 & 2 & 2 \\ 1    \\ \end{ytableau}} \\ \hline

\begin{tikzpicture}
	\node at (0,0) {$1$};
	\node at (1.2,0) {$2$};
	\node at (.6,.5) {$2$};
	\node at (.3,.15) {$\nearrow$};
	\node at (1.5,.15) {$\nearrow$};
	\node at (.3,.65) {$\nwarrow$};
	\node at (.9,.65) {$\nearrow$};
	\end{tikzpicture}
	& \raisebox{4mm}{$u^2wX_1^2X_2^3$}
	& \raisebox{5mm}{\footnotesize \begin{ytableau}  2 & 2 & 2 \\ 1 &1    \\ \end{ytableau}} \\ \hline

\begin{tikzpicture}
	\node at (0,0) {$1$};
	\node at (1.2,0) {$2$};
	\node at (.6,.5) {$2$};
	\node at (.3,.15) {$\nearrow$};
	\node at (1.5,.15) {$\nearrow$};
	\node at (.9,.65) {$\nearrow$};
	\end{tikzpicture}
	& \raisebox{4mm}{$u^3X_1^3X_2^3$}
	& \raisebox{5mm}{\footnotesize \begin{ytableau}  2 & 2 & 2 \\ 1 & 1 & 1   \\ \end{ytableau}} \\ \hline
\end{tabular}
\end{tabular}
\end{center}
\bigskip

 One of the surprising observations we made in this work is the following: Suppose we are able to construct an explicit bijection between AMTs and SBCSPPs that also preserves the $n+3$ statistics. If, in addition, the bijection is (in a straightforward sense that is defined below) compatible with the sign-reversing involutions from the proof of Proposition~\ref{prop: SBCSPP and DPP}, then we would have  as an immediate corollary an explicit bijection between the original ASMs and DPPs, and such a bijection would not make use of the involution principle of Garsia and Milne as long as the bijection for the extended objects did not. However, we are able to argue that such a compatible bijection does not exist at all (or at least that yet another twist might be necessary in the notion of compatibility; see Proposition~\ref{prop: AMT - SBCSPP} for the precise statement), which might be an explanation for the obstacles we have encountered in the last $40$ years.

To see this, let us start with the case $n=2$, where things still work. The two DPP-SBCSPPs (see the proof of Proposition~\ref{prop: SBCSPP and DPP} for the definition) are 
$$
\emptyset \quad \text{and} \quad  \tiny \begin{ytableau}  2 & 2 & 2 \\ 1   \\ \end{ytableau}, 
$$
and the above four weight-preserving bijections coincide on them. Therefore each of these bijections induce the same bijection 
 between these DPP-SBCSPPs and (plain) monotone triangles with bottom row $(1,2)$, simply by ignoring the arrows in the corresponding arrowed monotone triangles. One might then hope that a bijection between monotone triangles and DPP-SBCSPPs, and therefore DPPs, could be induced by a weight-preserving bijection between AMTs and SBCSPPs in a similar way (i.e., by ignoring arrows) for larger $n$, however the following proposition shows that this is not possible.

\begin{prop}
\label{prop: AMT - SBCSPP}
Let $n \geq 3$. Then there exists no weight-preserving bijection between arrowed monotone triangles with bottom row $1,\ldots,n$ and SBCSPPs of order $n$, which induces a bijection between  monotone triangles with bottom row $1,\ldots,n$ and DPP-SBCSPPs by restricting to DPP-SBCSPPs and by ignoring the arrows in the corresponding arrowed monotone triangles.
\end{prop}

\begin{proof}
The exponent of $X_1$ in the weight of an SBCSPP counts the occurrences of $1$'s. Since DPP-SBCSPPs can only have $1$'s weakly below the diagonal, and each column which has boxes weakly below the diagonal ends with a $1$, the exponent of $X_1$ is equal to the length $l$ of the Durfee square of the shape.

The shape of a DPP-SBCSPP is by definition $(a_1,\ldots,a_l|b_1,\ldots,b_l)$ in Frobenius notation with $a_i=b_i+1$ for all $1 \leq i \leq l$. We show $b_l \geq 1$. Assume $b_l=0$, then the bottom entry of the $l$-th column is on the diagonal, and, therefore, it has to be $1$. Since $a_l=b_l+1=1$, there is another box to the right of this $1$ and, by the weak decrease along rows, it has to be filled with $1$ as well. However, this contradicts that there are no $1$'s strictly above the diagonal.

The length of the first column of a DPP-SBCSPP is at most $n$, since the entries have to be strictly decreasing. Hence, the maximal value for $l$ is $n-1$ (using also the fact that $b_l \ge 1$) and it is only achieved by the shape $(n,n-1,\ldots,2|n-1,\ldots,1)$. The first $n-1$ columns have length $n$, and, therefore, these columns contain precisely the entries $n,n-1,\ldots,1$, in this order.
The last two column have length $n-1$. Since their last entries are above the diagonal, they must be at least $2$ and both columns contain the entries $n,n-1,\ldots,2$, in this order. Therefore there exists only one DPP-SBCSPP with $l=n-1$. Recall that this implies that its exponent of $X_1$ in the weight is $n-1$ and there exists no other DPP-SBCSPP with a greater exponent of $X_1$.

On the other hand, for all (plain) monotone triangles with $n$ as top entry, the exponents of $X_1$ in the weights of the associated arrowed monotone triangles are at least $n-1$. The number of monotone triangles with top entry $n$ is equal to the number of monotone triangles with bottom row $1,\ldots,n-1$, since the top entry $n$ forces all entries in the rightmost 
northwest-diagonal to be equal to $n$. For $n \geq 3$, there are at least $2$ of such monotone triangles. Hence not all of them can be reached through a weight-preserving bijection.
\end{proof}

Let us remark that in the case $n=3$, there are actually $71$ weights that appear precisely once (therefore it is clear what the hoped-for explicit weight-preserving bijection has to be for the 
corresponding arrowed monotone triangles and SBCSPPs), $31$ weights appear twice,  $6$ appear three times, $14$ appear four times, $6$ five times and $1$ six times.

\subsection*{Some guiding principles in our course of discovery}

We explain how we were naturally led to the extensions and the $n+3$ statistics because we believe that many more results can be established that way. In bijective combinatorics, new interesting objects and statistics are often discovered through experiments. This ``approach'' is certainly responsible for many unforeseeable and  thus very exciting developments in this field, which is particularly true for results involving alternating sign matrices, plane partitions and their relatives, and must not be underestimated. However, the approach that led to the results in this paper is different and in a sense more systematic (which should not be misinterpreted as if we thought that it is therefore also ``better''). Two ideas were the guiding principles:
\begin{enumerate}
\item In the past, to prove enumeration results on ASMs, it was often possible to deduce complicated, multivariate expressions that somehow ``included'' the enumerative quantities we are interested in, such as for instance in Zeilberger's first proof of the ASM theorem \cite{Zei96a} where he showed that the constant terms of certain multivariate rational functions are equal to the numbers of objects he was interested it. (To extract the constant terms from these expressions was then often another challenge, but this is not our point here.) It seems natural to think that these expressions include much more information on the objects than ``just'' their numbers, and this is what we discovered in this paper for a certain type of expression. (For expressions coming from the six-vertex model approach, it is evident from their definition that they are multivariate generating functions of ASMs, with certain unorthodox weights.) More concretely, we show that for a specific polynomial $r_n(Y_1,\ldots,Y_n)$ that has appeared in previous work as one whose constant term is the number of $n \times n$ ASMs, one can interpret $r_n(X_1-1,\ldots,X_n-1)$ as a generating function of $n \times n$ ASMs $A$ with weight $\omega_A(X_1,\ldots,X_n)$ for some weight function $\omega$ with $\omega_A(1,\ldots,1)=1$.\footnote{The polynomial $r_n(X_1-1,\ldots,X_n-1)$ is up to simple multiplicative factors a specialization of \eqref{opX} and $r_n$ has appeared in \cite[Proposition 10.1]{FR15}.}
From this, we gain ``control'' over a linear number $n$ of statistics (roughly speaking, the exponents of the $X_i$). Since other multivariate polynomials or rational functions that have been identified as ones whose constant terms give ASM numbers (including refined ones) in previous work (e.g., in \cite{Zei96a}) were constructed in a similar, recursive way as $r_n(Y_1,\ldots,Y_n)$, it is plausible that something analogous could work for them.

\item  The second guiding idea is the following. Enumerative results on ASMs that were deduced from these constant term expressions also include the equinumerosity with DPPs, and there exist proofs of such results that directly connect (by calculation) ``ASM quantities'' to ``DPP quantities''.\footnote{Such proofs can be found in \cite{Fis07,Fis16,partII}. They are mostly formulated in terms of the ``operator formula'' from \cite{Fis06}. However, the operator formula has an equivalent constant term formulation, see, e.g., Corollary 3.1 in \cite{Fis16}.} We have been able to find such a (new) connecting calculation that allowed us to extend to the multivariate setting in Section~\ref{proofmain2}, in particular we view Lemma~\ref{general} as an important step as it provides a certain bialternant.
 The extending objects on the DPP side were discovered by some kind of ``reverse engineering'', i.e., through interpreting the multivariate expression we have obtained after transforming the multivariate ``ASM expression'' into a  multivariate ``DPP expression". Here, Lemma~\ref{dettodet} was helpful to transform the bialternant into a Jacobi-Trudi 
 type determinant, which could then be interpreted combinatorially using the Lindstr\"om-Gessel-Viennot lemma \cite{Lin73,GesVie85,GesVie89}. In fact, this last step has been quite an exciting purely combinatorial journey.
\end{enumerate}
We emphasize another moral of our story (which we think has several), as it might be of use in circumstances in bijective combinatorics when there seems to be an obstacle that needs to be overcome. On the ASM side, it was necessary to (somewhat) ``decompose'' an individual object into several (new) objects, that is, each ASM is assigned a set of the new, extending objects, such that, up to a multiplicative factor, the weight of an ASM is the sum of the weights of the objects in the set assigned to it. We were naturally led to these new objects and sets (this is another nice combinatorial journey sketched below), because the weights of the new objects are monomials in the variables (which is not the case for ASMs themselves) and thus allows us to directly interpret the exponents of the variables as statistics. This is  somewhat connected to the concept of randomized bijections (as it appears, e.g., in \cite{AyyerFischer20}), which extends the concept of bijective proofs.

\bigskip

\subsection*{Outline}

The rest of the paper is organized as follows. In Section~\ref{ASM}, we elaborate more on the extension on the ASM-side: we will see how one is naturally led to arrowed monotone triangles and how to extend them naturally to arbitrary integer sequences $(k_1,\ldots,k_n)$ at the cost of having to work with signed sets. In Section~\ref{dpp}, we provide several objects that are equivalent to SBCSPPs, in particular certain families of non-intersecting lattice paths that will be crucial for our proof of Theorem~\ref{main2}. We will also see how these paths naturally extend the classical families of non-intersecting lattice that are known to be in (easy) bijective correspondence \cite{Lal03,Kra06}
 with DPPs (which is where the above sketched sign-reversing involution for SBCSPP comes from when specializing as in \eqref{specialization}). In Section~\ref{constantterm}, we provide a proof of a generalization of Theorem~\ref{main1} (our first main result), while in Section~\ref{decorated} we explain how to obtain a much more general version that deals with more general decorations. In Section~\ref{proofmain2}, we provide the proof of Theorem~\ref{main2} (our second main result). This proof relies heavily on Theorem~\ref{main1}. More precisely, 
 we start with the generating function \eqref{opX} from that theorem and apply a number of manipulations before we are able to use the Lindstr\"om-Gessel-Viennot lemma to interpret the expression combinatorially as some families of non-intersecting lattice paths for which we have seen in Section~\ref{dpp} that they correspond to SBCSPPs. An important manipulation is the one in Lemma~\ref{general} that essentially transforms \eqref{opX} into a bialternant and then a standard lemma (Lemma~\ref{dettodet}) can be used to obtain a Jacobi-Trudi type expression. After that a number of further manipulations (mostly matrix multiplications) are necessary to obtain the correct determinant.

\section{Extension on the ASM-side} 
\label{ASM}

In this section, we elaborate on the various ``ASM-objects''. We explain the genesis of arrowed monotone triangles which is through a certain ``non-monomial'' multivariate weight on (plain) monotone triangles. This weight is natural in the sense that it has its origin in the recursion that ``underlies'' monotone triangles  and the operator formula \eqref{op} from earlier work \cite{Fis06}. Roughly speaking, arrowed monotone triangles are introduced to decompose the non-monomial weights into monomial weights so that the exponents of the variables can then be interpreted as statistics. A variant of arrowed monotone triangles are down-arrowed monotone triangles. They are useful in our argument to connect the generating function of monotone triangles to the generating function of arrowed monotone triangles in Proposition~\ref{downup}. Another advantage of arrowed monotone triangles is that we can generalize them to arbitrary integer valued bottom rows (not necessarily increasing). A second generalization goes into the direction that we allow entries with no decoration, see Section~\ref{eamt}, and, in particular, Theorem~\ref{robbins}.
 A further generalization concerning the decorations is provided much later in the paper in Section~\ref{decorated}.
The next table gives an overview of the various related ASM-objects.

\medskip

\begin{center}
\begin{tabular}{|c|c|c|c|} \hline 
objects & weight & enumeration & definition\\
\hline
monotone triangles (MTs) & non-monomial & Theorem~\ref{robbins0} & Section~\ref{sectmain}, Eq.~\ref{MTweight}\\
down-arrowed MTs & monomial & Theorem~\ref{robbins0} & Section~\ref{damt} \\
arrowed MTs & monomial & Theorem~\ref{main1} & Definition~\ref{ramt}\\
extended arrowed MTs  & monomial & Theorem~\ref{robbins} & Section~\ref{eamt} \\ \hline
\end{tabular}
\end{center}

\subsection{A multivariate weight on monotone triangles} 
\label{damt}
Consider for an $n \times n$ ASM $A=(a_{i,j})_{1 \le i,j \le n}$ the matrix obtained by adding to each entry all the entries that are in the same column above, i.e., $S=(\sum\limits_{i'=1}^{i} a_{i' j})_{1 \le i, j \le n}$. It is not hard to see that this always results in a $\{0,1\}$-matrix with $i$ occurrences of $1$'s in the $i$-th row. The corresponding {monotone triangle} is obtained by recording row by row the columns of the $1$'s in $S$ in $n$ centered rows. An example is given next. 
$$   \left( \begin{tabular}{rrrrr}
0 & 1 & 0 & 0 & 0 \\
1& -1 & 0 & 1 & 0 \\
0 & 1 & 0 & -1 & 1 \\
0 & 0 & 1 & 0 & 0 \\
0 & 0 & 0 &  1 & 0
\end{tabular} \right)  \quad \Rightarrow \quad 
 \left( \begin{tabular}{rrrrr}
0 & 1 & 0 & 0 & 0 \\
1& 0 & 0 & 1 & 0 \\
1 & 1 & 0 & 0 & 1 \\
1 & 1 & 1 & 0 & 1 \\
1 & 1 & 1 &  1 & 1
\end{tabular} \right) \quad \Rightarrow \quad 
\begin{array}{ccccccccc}
 & & & & 2 & & & &  \\
 & & & 1 & & 4 & & &  \\
 & & 1 & & 2 & & 5 & &  \\
 & 1 & & 2 & & 3 & & 5 &  \\
 1 & & 2 & & 3 & & 4 & & 5
\end{array}
$$
This procedure establishes a bijection between $n \times n$ ASMs and monotone triangles with bottom row $1,2,\ldots,n$. 
It was shown in  \cite{Fis06} that the number of monotone triangles with bottom row $k_1,k_2,\ldots,k_n$ is\footnote{In \cite{Fis06}, the formula appears as 
$\prod_{1 \le p < q \le n} \left(  \id + \e_{k_p} \Delta_{k_q} \right) \prod_{1 \le i < j \le n } \frac{k_j-k_i}{j-i}$ with $\Delta_x=\e_x - \id$, but 
rewriting $\id + \e_{k_p} \Delta_{k_q} = \e_{k_q} (\e_{k_p} + \e_{k_q}^{-1} - \e_{k_p} \e_{k_q}^{-1})$ and noting that 
$\prod_{1 \le p < q \le n} \e_{k_q} \prod_{1 \le i < j \le n } \frac{k_j-k_i}{j-i} = \prod_{1 \le i < j \le n } \frac{k_j-k_i+j-i}{j-i}$ shows that the two formulas are equivalent.}
\begin{equation} 
\label{op} 
\prod_{1 \le p < q \le n} \left(  \e_{k_p} +  \e_{k_q}^{-1} -  \e_{k_p} \e_{k_q}^{-1} \right) \prod_{1 \le i < j \le n } \frac{k_j-k_i+j-i}{j-i}.
\end{equation}

A crucial fact for what follows is that $\prod_{1 \le i < j \le n } \frac{k_j-k_i+j-i}{j-i}$ is the number of semistandard tableaux of shape $(k_n,\ldots,k_1)$ and entries at most $n$. Note that such semistandard tableaux  are in easy bijective correspondence with Gelfand-Tsetlin patterns with bottom row $k_1,\ldots,k_n$, which are defined just as monotone triangles with the exception that the condition on the strict increase along rows is dropped (see \cite[p.\ 313]{Sta99}).
This number is the specialization of the Schur polynomial $s_{(k_n,k_{n-1},\ldots,k_1)}(X_1,\ldots,X_n)$ at $(X_1,\ldots,X_n)=(1,\ldots,1)$ (see \cite[p. 375, Eq (7.105)]{Sta99}) and the Schur polynomial can be interpreted as a multivariate generating function of Gelfand-Tsetlin patterns (or, equivalently, of semistandard tableaux, see \cite[Sect.\ 7.10]{Sta99}).
Hence, this brings us to the obvious question whether we can also ``introduce'' the variables $X_1,\ldots,X_n$ in the operator formula, by replacing $\prod_{1 \le i < j \le n } \frac{k_j-k_i+j-i}{j-i}$ in the operator formula with $s_{(k_n,k_{n-1},\ldots,k_1)}(X_1,\ldots,X_n)$ and possibly modifying the operator polynomial. Our first main result will show that this is indeed possible. In order to formulate it, we need a number of statistics.

Each entry $m_{i,j}$ of a monotone triangle $M=(m_{i,j})_{1 \le j \le i \le n}$  that is not in the bottom row is in precisely one of the following categories.
\begin{itemize} 
\item $m_{i+1,i} < m_{i,j} < m_{i+1,j+1}$: in this case, $m_{i,j}$ is said be \emph{special}.
\item $m_{i,j}=m_{i+1,j}$: in this case, $m_{i,j}$ is said to be \emph{left-leaning}.
\item $m_{i,j}=m_{i+1,j+1}$: in this case, $m_{i,j}$ is said to be \emph{right-leaning}.
\end{itemize} 
In our running example
$$
\begin{array}{ccccccccccccccccc}
  &   &   &   &   &   &   &   & 4 &   &   &   &   &   &   &   & \\
  &   &   &   &   &   &   & 2 &   & 5 &   &   &   &   &   &   & \\
  &   &   &   &   &   & 2 &   & 3 &   & 5 &   &   &   &   &   & \\
  &   &   &   &   & 1 &   & 3 &   & 4 &   & 6 &   &   &   &   & \\
  &   &   &   & 1 &   & 2 &  &   3 &   & 5   &  & 6  &   &   &   & \\
  &   &   & 1 &   & 2 &   & 3 &   & 4 &   & 5 &   & 6 &   &   &
\end{array},
$$ 
the $2$ in row $3$ is special, while the $3$ in row $3$ is left-leaning and $5$ in row $2$ is right-leaning.
We define the following statistics for $i=1,\ldots,n-1$: 
\begin{equation} 
\begin{aligned} 
s_i(M) & = \# \text{ of special entries in  row $i$} & s(M) &=  \# \text{ of all special entries} \\
l_i(M) & = \# \text{ of left-leaning entries in  row $i$} & l(M) &=  \# \text{ of all left-leaning entries} \\
r_i(M) & = \# \text{ of right-leaning entries in  row $i$} & r(M) &=  \# \text{ of all right-leaning entries} 
\end{aligned}
\end{equation} 
For our example, this gives the following.
\begin{align*}
(s_1(M),s_2(M),s_3(M),s_4(M),s_5(M))&= (1,0,2,1,0), & s(M) &= 4, \\
(l_1(M),l_2(M),l_3(M),l_4(M),l_5(M))& = (0,1,1,1,3), & l(M) &= 6, \\
(r_1(M),r_2(M),r_3(M),r_4(M),r_5(M))& = (0,1,0,2,2), & r(M)& =5.     
\end{align*} 
Further, we let $s_0(M)=r_0(M)=l_0(M)=0$ and define now for $i=1,2,\ldots,n$
\begin{equation} 
d_i(M) = \sum_{j=1}^{i} m_{i,j} - \sum_{j=1}^{i-1} m_{i-1,j} + r_{i-1}(M) - l_{i-1}(M).
\end{equation} 
In our example, we obtain 
$$
(d_1(M),d_2(M),d_3(M),d_4(M),d_5(M),d_6(M))=(4,3,3,3,4,3).
$$
For monotone triangles with bottom row $1,2,\ldots,n$, the statistics $l,r,s$ correspond to the  \emph{complementary inversion number}, one version of the \emph{inversion number}, and the number of $-1$'s, respectively. The statistic $s$ was introduced in \cite[Section 5]{MilRobRum83}, the 
appropriate type of the inversion number appeared for the first time in \cite[Eq. (18)]{lambda} (note that 
in \cite{MilRobRum83} a slight variation is used) and the complementary inversion number appeared in 
\cite[Sec. 2.1]{Fis18}. For more details see also \cite[Section 3]{FSA22}.

We introduce the following weight on monotone triangles
\begin{equation}
\label{MTweight}
\W_0(M) = u^{r(M)} v^{l(M)} \prod_{i=1}^{n} X_i^{d_i(M)}  \left(w + u X_i + v X_i^{-1} \right)^{s_{i-1}(M)}. 
\end{equation}
Note that there is an analogy to the classical weight $\prod_{i=1}^n X_i^{\sum_{j=1}^{i} a_{i,j} - \sum_{j=1}^{i-1} a_{i-1,j}}$ for Gelfand-Tsetlin patterns $(a_{i,j})_{1 \le j \le i \le n}$ , which is 
used in the definition of the Schur polynomial.
When setting 
\begin{equation}
\label{spec}
u=v=1, w=-1 \text{ and } (X_1,\ldots,X_n)=(1,\ldots,1), 
\end{equation}
the weight specializes to $1$.
The weight of our example is
$$
u^5 v^6 X_1^4 X_2^3 X_3^3 X_4^3 X_5^4 X_6^3 
\left(w + u X_2 + v X_2^{-1} \right) \left(w+ u X_4 + v X_4^{-1} \right)^2 
\left(w + u X_5 + v X_5^{-1} \right).
$$

In order to state a formula for the generating function of monotone triangles with fixed bottom row, we need an extension of Schur polynomials to sequences of integers $\lambda=(\lambda_1,\ldots,\lambda_n)$ that are not necessarily partitions. If $\lambda_i-i=\lambda_j-j$ for some $1 \leq i < j \leq n$, we define $s_{\lambda}(X_1,\ldots,X_n)=0$. Otherwise there exists a unique permutation $\sigma$ on $\{1,2,\ldots,n\}$ with
\[
\lambda_{\sigma(1)}-\sigma(1) > \lambda_{\sigma(2)}-\sigma(2)  > \ldots > \lambda_{\sigma(n)}-\sigma(n). 
\]
Let $p$ be a non-negative integer such that $\lambda_{\sigma(n)} + n - \sigma(n) + p \ge 0$. Then we define 
\begin{equation} 
s_{\lambda}(X_1,\ldots,X_n) = \sgn(\sigma) s_{\mu}(X_1,\ldots,X_n) \prod_{i=1}^n X_i^{-p},
\end{equation} 
with 
$$
\mu = (\lambda_{\sigma(1)} + 1 - \sigma(1) +p,\lambda_{\sigma(2)} + 2-  \sigma(2)+p,\ldots,\lambda_{\sigma(n)} + n - \sigma(n)+p).
$$
Note that, although the choice of $p$ is not unique, the extension of the Schur polynomial is well-defined since 
$s_{\nu+p}(X_1,\ldots,X_n) \prod_{i=1}^n X_i^{-p}$ is independent of the non-negative integer $p$ for any partition $\nu$, where $\nu+p$ is defined as coordinatewise addition with $p$.
\medskip

We have the following explicit formula for the generating function of monotone triangles with fixed bottom row $k_1,k_2,\ldots,k_n$. 

\begin{theorem} 
\label{robbins0}
The generating function of monotone triangles $M$ with bottom row $k_1 < \ldots < k_n$ with respect to the weight $\W_0(M)$ is 
\begin{equation} 
\prod_{1 \le p < q \le n} \left( u  \e_{k_p} + v  \e_{k_q}^{-1} + w  \e_{k_p} \e_{k_q}^{-1} \right) \\
s_{(k_n,k_{n-1},\ldots,k_1)}(X_1,\ldots,X_n). 
\end{equation} 
In other words, the generating function is the same as in Theorem~\ref{main1}, except for the first product there.
\end{theorem}

In Section~\ref{implies} we prove that this theorem is a consequence of Theorem~\ref{robbins}.

It is inconvenient that the weight of a monotone triangle is not a monomial in $u,v,w,X_1,\ldots,X_n$, because otherwise the exponents of 
$u,v,w,X_1,X_2,\ldots,X_n$ could simply be interpreted as statistics of the objects.  This can be changed by considering a slight modification: We will associate with each monotone triangle $3^{s(M)}$ objects and introduce weights for them, such that these weights are indeed monomials in $u,v,w, X_1,\ldots,X_n$ and the sum of weights of the $3^{s(M)}$ 
objects is the weight of the corresponding monotone triangle. 

The modification is straightforward: Each special entry of a monotone triangle can ``choose'' between being left-leaning, right-leaning or central; left-leaning is indicated for special entries by  $\swarrow$, right-leaning is indicated by $\searrow$, and central is indicated by $\downarrow$.
$$
\begin{array}{ccccccccccccccccc}
  &   &   &   &   &   &   &   & {_\swarrow 4} &   &   &   &   &   &   &   & \\
  &   &   &   &   &   &   & 2 &   & 5 &   &   &   &   &   &   & \\
  &   &   &   &   &   & 2  &   & 3 &   & {5_\searrow} &   &   &   &   &   & \\
  &   &   &   &   & 1 & \downarrow \atop \phantom{\downarrow} & 3 &   & {_ \swarrow 4} &   & 6 &   &   &   &   & \\
  &   &   &   & 1 &   & 2 &  &   3 &   & 5   &  & 6  &   &   &   & \\
  &   &   & 1 &   & 2 &   & 3 &   & 4 &   & 5 &   & 6 &   &   &
\end{array}
$$ 
Monotone triangles where the special entries are decorated in that way are said to be \emph{down-arrowed monotone triangles} (to distinguish them from arrowed monotone triangles). Concerning the weight with the conventions we had for (plain) monotone triangles in the 
sense that now also special entries that are left-leaning in row $i-1$ contribute to the corresponding $l_{i-1}(M)$, and thus to the power of $u$ and ``negatively'' to the power of $X_{i}$, while special entries that are right-leaning in row $i-1$ contribute to the corresponding $r_{i-1}(M)$, and thus to the power of $v$, and ``positively'' to the power of $X_{i}$. The total number of central entries is denoted by $c(M)$, and is the exponent of $w$. Note that in the above example the entry $2$ in row $3$ is central. Thus its weight is
$$
u^6 v^8 w X_1^4 X_2^2 X_3^3 X_4^4 X_5^3 X_6^3. 
$$
Since this mimics what happens when expanding the product of the second factors in \eqref{MTweight}, Theorem~\ref{robbins0} provides the generating function of down-arrowed monotone triangles with fixed bottom row.

\subsection{Extended arrowed monotone triangles} 
\label{eamt}
There are certain similarities between the down-arrowed monotone triangles and the \emph{arrowed monotone triangles} as described in Definition~\ref{ramt}: Each (plain) monotone triangle is assigned a set of these special arrowed monotone triangles (namely those that have the monotone triangle as the underlying one when simply ignoring the arrows) such that the weight of the monotone triangle differs by the sum of weights of the arrowed monotone triangles assigned to it only by the factor in \eqref{factor}). The weight assigned to a fixed arrowed monotone triangle is also a monomial in $u,v,w, X_1,\ldots,X_n$. The advantage of arrowed monotone triangles is that they allow certain generalizations, including the extension to integer sequences $k_1,k_2,\ldots,k_n$ that are not necessarily increasing. Also arrowed monotone triangles will be crucial for our proof of Theorem~\ref{robbins0}.

This extension to not necessarily increasing bottom rows requires the notion of signed sets, which is a pair of disjoint sets $\u S=(S^+,S^-)$. The size of a signed set is 
\begin{equation} 
| \u S| = |S^+| - |S^-|.
\end{equation} 
Signed sets can obviously also be considered as ordinary sets with a sign function, and the size is then the generating function with respect to the sign. More generally, for given weights on the elements of $\u S$, the generating function  is defined as 
\begin{equation} 
\sum_{s \in S^+} \W(s)- \sum_{s \in S^-} \W(s).
\end{equation} 
We will use \emph{signed intervals}: for $a, b \in \mathbb{Z}$, we let 
\begin{equation} 
\si{a}{b} = \begin{cases} ([a,b],\emptyset) & a \le b \\
                                       (\emptyset,[b+1,a-1]) & b+1 \le  a-1 \\
                                       (\emptyset,\emptyset) & b+1 = a
                \end{cases}.                       
\end{equation} 
The sign of a signed interval is positive if and only if the negative part is empty.

An \emph{extended arrowed monotone triangle} (AMT) is a triangular array of the following form
$$
\begin{array}{ccccccccccccccccc}
  &   &   &   &   &   &   &   & a_{1,1} &   &   &   &   &   &   &   & \\
  &   &   &   &   &   &   & a_{2,1} &   & a_{2,2} &   &   &   &   &   &   & \\
  &   &   &   &   &   & \dots &   & \dots &   & \dots &   &   &   &   &   & \\
  &   &   &   &   & a_{n-2,1} &   & \dots &   & \dots &   & a_{n-2,n-2} &   &   &   &   & \\
  &   &   &   & a_{n-1,1} &   & a_{n-1,2} &  &   \dots &   & \dots   &  & a_{n-1,n-1}  &   &   &   & \\
  &   &   & a_{n,1} &   & a_{n,2} &   & a_{n,3} &   & \dots &   & \dots &   & a_{n,n} &   &   &
\end{array},
$$
where each $a_{i,j}$ is an integer decorated with an element from $\{\nwarrow, \nearrow, \nenwarrow,\emptyset \}$ and the following is satisfied for each integer $a$ not in the bottom row: Suppose $b$ is the 
southwest-neighbor of $a$ and $c$ is the 
southeast-neighbor of $a$, respectively, i.e.,
$$
\begin{array}{ccc}
&a& \\
b&&c
\end{array}.
$$
Depending on the decoration of $b, c$, denoted as $\dec(b), \dec ( c )$, respectively, we need to consider four cases:
\begin{itemize}
\item If $\dec(b) \in \{\nwarrow,\emptyset\}$ and $\dec(c) \in \{\nearrow, \emptyset\}$: $a \in \si{b}{c}$
\item If $\dec(b) \in \{\nwarrow,\emptyset\}$ and $\dec(c) \in \{\nwarrow, \nenwarrow\}$: $a \in \si{b}{c-1}$
\item If $\dec(b) \in \{\nearrow, \nenwarrow \}$ and $\dec(c) \in \{\nearrow,\emptyset\}$: $a \in \si{b+1}{c}$
\item If $\dec(b) \in \{\nearrow, \nenwarrow \}$ and $\dec(c) \in \{\nwarrow, \nenwarrow\}$: $a \in \si{b+1}{c-1}$
\end{itemize}
(Note that Definition~\ref{ramt} is equivalent to this if $a_{n,1} < a_{n,2} < \ldots < a_{n,n}$ because then all occurring intervals 
$\si{b(+1)}{c(-1)}$ are positive and all rows are strictly increasing.)
Extended arrowed monotone triangles are considered as a signed set as follows: For a fixed 
extended arrowed monotone triangle $(a_{i,j})_{1 \le j \le i \le n}$, each negative interval 
$\si{a_{i,j}(+1)}{a_{i,j+1}(-1)}$, with $i \ge 2$ and $j < i$, contributes a multiplicative 
$-1$, choosing $a_{i,j}+1$ iff $\dec(a_{i,j}) \in 
\{\nearrow, \nenwarrow \}$  and $a_{i,j}$ otherwise, and $a_{i,j+1} - 1$ iff 
$\dec(a_{i,j+1}) \in \{ \nwarrow, \nenwarrow \}$ and $a_{i,j+1}$ otherwise.
 
We associate the following weight to a given extended arrowed monotone triangle $A=(a_{i,j})_{1 \le j \le i \le n}$:
\begin{equation}
\W(A) = \wne^{\# \nearrow} \wnw^{\# \nwarrow} \wnenw^{\# \nenwarrow} \wnone^{\# \emptyset}  \prod_{i=1}^{n} X_i^{\sum_{j=1}^i a_{i,j}  - \sum_{j=1}^{i-1} a_{i-1,j} + \# \nearrow \text{in row $i$ } - \# \nwarrow \text{in row $i$ }} 
\end{equation}
We refrain from using $u,v,w$ and use $\wne, \wnw, \wnenw$ instead to prepare for a further generalization in Section~\ref{decorated}.

We have the following explicit formula for the generating function of extended arrowed monotone triangles with fixed bottom row $k_1,k_2,\ldots,k_n$. It is a generalization of Theorem~\ref{main1} and its proof is provided in Section~\ref{constantterm}.

\begin{theorem} 
\label{robbins}
The generating function of extended arrowed monotone triangles with bottom row $k_1,\ldots,k_n$ is 
\begin{multline}
\prod_{i=1}^{n} (\wne \cdot X_i + \wnw \cdot X_i^{-1} + \wnenw + \wnone) \\ 
\prod_{1 \le p < q \le n} \left( \wne \cdot \e_{k_p} + \wnw \cdot \e_{k_q}^{-1} + \wnenw \cdot \e_{k_p} \e_{k_q}^{-1} + \wnone \cdot \id  \right) \\
s_{(k_n,k_{n-1},\ldots,k_1)}(X_1,\ldots,X_n). 
\end{multline} 
\end{theorem}

\subsection{Theorem~\ref{robbins} implies Theorem~\ref{robbins0}} 
\label{implies}

We consider extended arrowed monotone triangles with strictly increasing bottom row that have no $\emptyset$'s assigned as decoration, which clearly corresponds to setting $\wnone = 0$ in Theorem~\ref{robbins} (and also to the arrowed monotone triangles of Section~\ref{sectmain}). Neglecting the decorations of such extended arrowed monotone triangles gives monotone triangles, in fact, all monotone triangles can be obtained that way (but, in general, each of them will appear multiple times, with several possible decorations). 

The purpose of this section is the following proposition. 

\begin{prop}
\label{downup}
For a given monotone triangle, the sum of weights of all associated arrowed monotone triangles equals the weight of the fixed monotone triangle (as defined in Section~\ref{damt}) up to the factor 
\begin{equation}
\label{factor} 
\prod_{i=1}^{n} (u  X_i + v  X_i^{-1}+w)
\end{equation}
when setting $\wne = u, \wnw = v, \wnenw = w$. 
\end{prop}

\begin{proof}
In order to show this, it suffices to consider the situation of two consecutive rows.
Suppose $k_1 < k_2 < \ldots < k_i$ and that $m_1,\ldots,m_{i-1}$ is interlacing in the sense that 
$$
k_1 \le m_1 \le k_2 \le m_2 \le k_3 \le \ldots \le m_{i-1} \le k_i
$$
with $m_1 < m_2 < \ldots < m_{i-1}$. One should think of $(m_1,\ldots,m_{i-1})$ and $(k_1,\ldots,k_i)$ as two consecutive rows in a monotone triangle. For a monotone triangle $M$, the weight $\W_0(M)$ can be ``structured'' according to the contribution of pairs of consecutive rows, i.e., the contribution of the  pair of rows $i-1$ 
and $i$, $1 \le i \le n$, is 
$$
u^{r_{i-1}(A)} v^{l_{i-1}(A)} X_i^{d_i(A)} (w+ u X_i + v X_i^{-1})^{s_{i-1}(A)}.
$$ 
In this sense, the contribution of the rows $(m_1,\ldots,m_{i-1})$ and $(k_1,\ldots,k_i)$ to the weight of a monotone triangle is as follows:
\begin{equation} 
\label{cont}
(u X_i)^{\# \text{ of right-leaning } m_j} (v X_{i}^{-1})^{\# \text{ of left-leaning } m_j} 
\left( w + u X_i + v X_i^{-1} \right)^{\# \text{ of special } m_j} X_i^{k_1 + \ldots + k_i - m_1 - \ldots - m_{i-1}} 
\end{equation} 

We determine all possible arrow assignments to $k_1,\ldots,k_i$ in arrowed monotone triangles. We say that $m_j$ is NL (for non left-leaning) if 
$m_j \not= k_j$, and it is NR (for non right-leaning) if $m_j \not= k_{j+1}$. 
\begin{itemize} 
\item If $m_j$ is NL, then we may  assign $\nearrow$ to $k_j$.
\item If $m_{j-1}$ is NR, then we may  assign $\nwarrow$ to $k_j$. 
\item If $m_j$ is NL and $m_{j-1}$ is NR, we may assign $\nenwarrow$ to $k_j$.
\item Each $k_j$ has to be assigned at least one arrow (which is possible since $m_{j-1} = k_j = m_j$ is excluded). 
\end{itemize} 

As for arrowed monotone triangles, the contribution of the pair of the rows $i-1$ and $i$, $1 \le i \le n$, is 
$$
(u X_i)^{\# \nearrow \text{ in row } i} (v X_i^{-1})^{\# \nwarrow \text{ in row } i}  w^{\# \nenwarrow \text{ in row } i} 
X_i^{\text{sum of entries in row $i$} - \text{sum of entries in row $i-1$}}
$$ 
We need to determine the sum of weights over all arrow assignments for a particular choice of $(m_1,\ldots,m_{i-1})$ and 
$(k_1,\ldots,k_i)$. 
Observe that we have a choice for the decoration of $k_j$ if and only if one of the following is satisfied.
\begin{itemize} 
\item $j \not =1,i$, and $m_{j-1}$ is NR and $m_{j}$ is NL.
\item $j=1$ and $m_1$ is NL.
\item $j=i$ and $m_{i-1}$ is NR.
\end{itemize} 
In these cases, we are completely free to choose between $\nearrow$, $\nwarrow$ and $\nenwarrow$. A $\nenwarrow$ contributes $w$, a  $\nearrow$ contributes $u X_i$, while $\nwarrow$ contributes $v X_i^{-1}$. In summary, the contribution of such a $k_j$ is 
$$
w+ u X_i + v X_i^{-1}.
$$
We say that $k_j$ is up-special in this case. If, on the other hand, $k_j$ is not up-special then either $m_{j-1}=k_j$ or $m_j=k_j$. In the first case, 
$k_j < m_j$, $k_j$ needs to be decorated with $\nearrow$, and we have a contribution of $u X_i$ to the weight, while in the second case $m_{j-1} < k_j$, $k_j$ 
needs to be decorated with $\nwarrow$, and we have a contribution of $v X_i^{-1}$ to the weight. Noting that $m_{j-1}$ is right-leaning in the first case, and $m_j$ is left-leaning in the second case, we have the following total weight
\begin{multline*} 
(u X_i)^{\# \text{ of right-leaning $m_j$}} (v X_{i}^{-1})^{\# \text{ of left-leaning $m_j$}} 
\left( w + u X_i + v X_i^{-1} \right)^{\# \text{ of up-special $k_j$'s}} \\ 
\times X_i^{k_1+\ldots+k_i-m_1-\ldots-m_{i-1}}.
\end{multline*} 
The assertion follows now after comparing with \eqref{cont} and observing that the number of special elements among the $m_j$ (i.e., those with $k_i < m_i < k_{i+1}$) is precisely one less than the number of up-special elements among the $k_j$: if we work through $k_1, m_1, k_2, \ldots,m_{n-1}, k_n$, then we will see that up-special $k_j$'s and special $m_j$'s alternate, starting and ending with up-special $k_j$'s, noting also that $k_1$ and $k_i$ are already up-special if $k_1 \not= m_1$ and $m_{i-1} \not= k_i$, respectively.
\end{proof}

\section{Extension on the DPP-side} 
\label{dpp} 

{In this section, we will elaborate on various versions of the extended ``DPP-objects''. As suggested by the titles of the sections, we start by recalling the classical ``DPP-paths'' that are in easy bijective correspondence with DPPs. Next we introduce the extended DPP-paths and justify why they are indeed an extension before we finally pass to plane partitions.
The following table gives an overview of the various objects.

\begin{table}[h]
\begin{tabular}{|c|c|c|} \hline 
objects & weight & definition \\
\hline
set-valued near-balanced CSPPs (SBCSPPs) & monomial & Definition~\ref{PP}\\
extended DPP paths & non-monomial & Section~\ref{dpp}\\
DPP pairs & monomial & Definition~\ref{DPPshape}\\
balanced CSPPs (BCSPPs) & non-monomial & Definition~\ref{BCSPP} \\ \hline
\end{tabular}
\end{table}


\subsection{Classical DPP-paths}
It is well-known \cite{Lal03,Kra06} that DPPs can be encoded by families of non-intersection lattice paths in the plane as follows (our convention differs from the one more frequently used by the reflection along the line $y=x$). Let $A_i=(0,i-2)$  and $E_i=(i,0)$, $i=2,\ldots,n$. We choose $S \subseteq \{2,\ldots,n\}$, and consider families of $|S|$ non-intersecting lattice paths starting in 
$(A_i)_{i \in S}$ and ending in $(E_i)_{i \in S}$, with unit right steps and unit down-steps. Here such families of paths are referred to as \emph{DPP paths}, or \emph{$n$-DPP paths} to stress $n$. The family that corresponds to the example from the introduction is provided in Figure~\ref{nilp}.

\begin{figure} 
\scalebox{0.5}{
\psfrag{1}{\color{dg} \Huge$1$}
\psfrag{2}{\color{dg} \Huge$2$}
\psfrag{3}{\color{dg} \Huge$3$}
\psfrag{4}{\color{dg} \Huge$4$}
\psfrag{5}{\color{dg} \Huge$5$}
\psfrag{6}{\color{dg} \Huge$6$}
\psfrag{7}{\color{dg} \Huge$7$}
\psfrag{8}{\color{dg} \Huge$8$}
\includegraphics{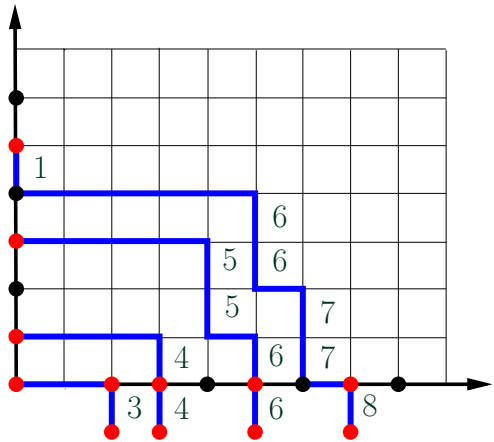}}
\caption{\label{nilp} Family of non-intersecting lattice paths that corresponds to a CSSPP of class $2$.} 
\end{figure}

The corresponding CSSPP of class $2$ is obtained as follows. We extend each path at the end point $E_i$ by an additional down-step. Now the entries of the CSSPP correspond to the down-steps as follows. Each path corresponds to a row of the CSSPP, with the topmost path corresponding to the top row. The entries are obtained as the distances of the down-steps from the $y$-axis by adding $1$ to each of them, where the bottommost down-step of a path corresponds to the first part of the corresponding row. (The usual convention, where the picture is reflected along the line $y=x$, would have clearly allowed us to speak essentially about the heights of the vertical steps, but we will see the advantage of the current convention in our setting.) 

\subsection{Our extended DPP-paths}
\label{extendedpaths} 
Our next goal is to define related objects and assign certain weights to them such that when we specialize the weights in a certain way, all objects have weights $0,1,-1$, thereby establishing a signed set (ignoring objects with weight $0$ then) such that the following is satisfied:
\begin{itemize} 
\item There is a very simple sign reversing involution on a subset of them that ``cancels'' all objects with negative weight (i.e, the negative part of the signed set is fully contained in this subset). 
\item There is a very simple bijection between the remaining objects (all of which have weight $1$) and the families of non-intersecting lattice paths as described above.
\end{itemize} 

We start with the observation that, in the above ``picture'', we can also add the starting point $A_1=(0,-1)$ and the end point 
$E_1=(1,0)$, because there is simply no path from $A_1$ to $E_1$ using the given step set, which implies that we can assume $1 \notin S$.
We now define the following different starting points and end points: $A'_i=(-i,i-2)$ and $E'_i=(i,-2)$, $i=1,\ldots,n$. Again we choose a subset 
$S \subseteq \{1,\ldots,n\}$, and consider families of $|S|$ non-intersecting lattice paths starting in 
$(A'_i)_{i \in S}$ and ending in $(E'_i)_{i \in S}$. 
Now the step set depends on the region we are in:
\begin{itemize} 
\item Strictly left of the $y$-axis, we allow unit up-steps and unit right-steps.
\item In the region $\{(x,y) | x \ge 0, y \ge -1 \}$, we allow unit down-steps and unit right-steps.
\item Weakly right of the $y$-axis, if we go below the $y=-1$ line (which we do when we make the final step to reach $E_i$), we can either perform a unit down-step or a diagonal-step $(1,-1)$. 
\item The paths should remain weakly below the line $y=n-2$. 
\end{itemize} 
An example for the case $n=8$ is provided in Figure~\ref{nilp2}. These families of paths are said to be the extended ($n$-)DPP paths. 

\begin{figure} 
\scalebox{0.4}{
\psfrag{1}{\color{dg} \Huge$1$}
\psfrag{2}{\color{dg} \Huge$2$}
\psfrag{3}{\color{dg} \Huge$3$}
\psfrag{4}{\color{dg} \Huge$4$}
\psfrag{5}{\color{dg} \Huge$5$}
\psfrag{6}{\color{dg} \Huge$6$}
\psfrag{7}{\color{dg} \Huge$7$}
\psfrag{8}{\color{dg} \Huge$8$}
\includegraphics{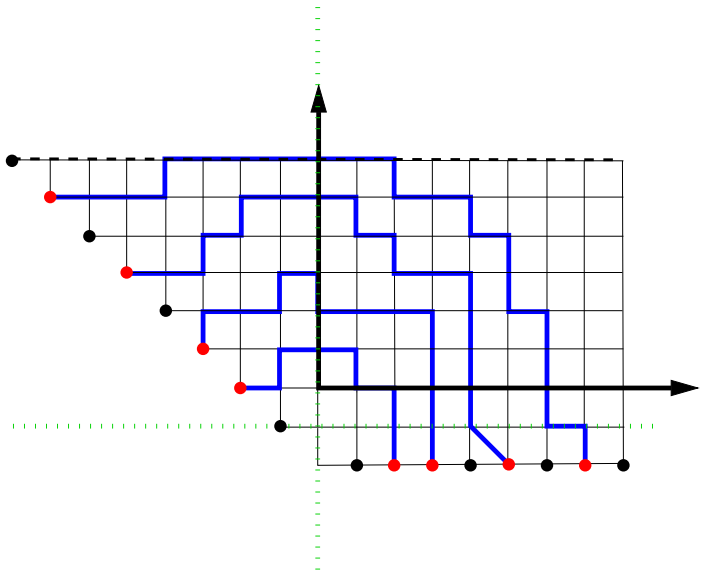}}
\caption{\label{nilp2} Extended DPP-paths} 
\end{figure}

We assign weights to such families by assigning weights to the individual steps; the weight of a path is then the product of all step weights and the weight of a family is the product of all path weights, multiplied by $v^{\binom{n+1}{2}}$ if we consider order $n$ objects. (Let us emphasize that \emph{any} subset of the $n$ pairs of starting points and end points gives rise to order $n$ objects.) The weight of an edge depends on the region, as indicated in Figure~\ref{weights}.
\begin{itemize} 
\item Strictly left of the $y$-axis, a unit up step that ends in $(p,q)$ has weight $X_{m} v^{-1} w + 1$ with $m=p+q+2$, while 
a right step that ends in $(p,q)$ has weight $X_{m} v^{-1}$ also with $m=p+q+2$.
\item In the region $\{(x,y) | x \ge 0, y \ge -1 \}$, a unit right step at height $m$ above the $x$-axis has weight $u X_{m+2}$, while down-steps have weight $1$.
\item As for the last step, we have weight $w$ for the diagonal steps and weight $1$ for the down-steps.
\end{itemize} 

\begin{figure} 
\scalebox{0.4}{
\psfrag{1}{\large $\bar 1$}
\psfrag{2}{\large $\bar 2$}
\psfrag{3}{\large $\bar 3$}
\psfrag{4}{\large $\bar 4$}
\psfrag{5}{\large $\bar 5$}
\psfrag{6}{\large $\bar 6$}
\psfrag{7}{\large $\bar 7$}
\psfrag{8}{\large $\bar 8$}
\psfrag{1u}{\large $1$}
\psfrag{2u}{\large $2$}
\psfrag{3u}{\large $3$}
\psfrag{4u}{\large $4$}
\psfrag{5u}{\large $5$}
\psfrag{6u}{\large $6$}
\psfrag{7u}{\large $7$}
\psfrag{8u}{\large $8$}
\psfrag{1r}{\large $\widehat{1}$}
\psfrag{2r}{\large $\widehat{2}$}
\psfrag{3r}{\large $\widehat{3}$}
\psfrag{4r}{\large $\widehat{4}$}
\psfrag{5r}{\large $\widehat{5}$}
\psfrag{6r}{\large $\widehat{6}$}
\psfrag{7r}{\large $\widehat{7}$}
\psfrag{8r}{\large $\widehat{8}$}
\psfrag{w}{\large $w$}
\includegraphics{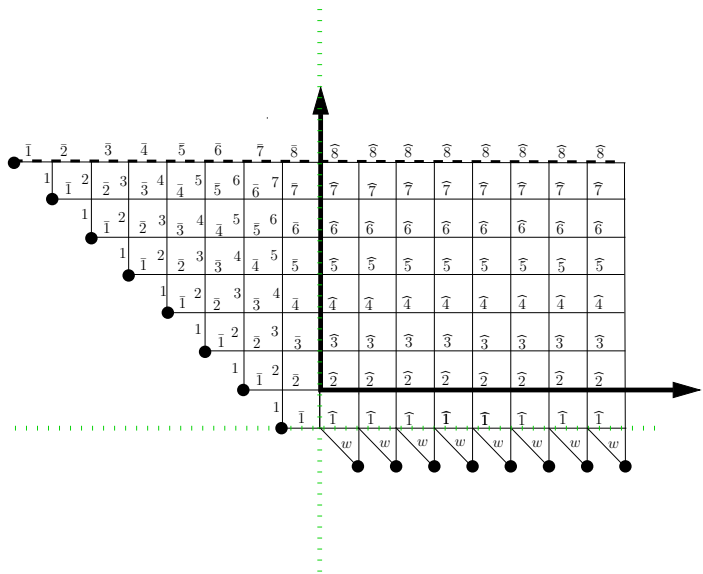}}
\caption{\label{weights} The weights of the extended DPP paths, where an integer  $i$ stands for  $X_i v^{-1} w + 1$, $\bar i$ stands for $X_i v^{-1}$, and  $\widehat{i}$ stands for $u X_i$; no label means that the step has weight $1$.} 
\end{figure}

\subsection{Why is it an extension?}
Next we will  argue why the specialization given in \eqref{spec} brings us to ordinary non-extended DPP paths: In this case all edge weights are $\pm 1, 0$. In particular, we have $0$ weight for all up steps left for the $y$-axis, so we can assume in this case that  
we have only right steps from $A'_i$ until we reach the $y$-axis, which will then be reached ``at'' $A_i$, and so we can simply let the paths start there. What remains are edges with weights 
$\pm 1$, where the edges with weight $-1$ are precisely the (possible) diagonal-steps that reach $E'_i$. However, such paths can actually be identified with paths that reach $E_i$ by a down-step that is preceded by a horizontal step. Thus we can also assume that paths with diagonal-steps cancel such paths (from right to left), so that all paths end with two down-step. Thus it suffices to consider paths that end in $E_i$ instead of $E'_i$, which correspond to classical DPP paths. 

We also want to introduce a slight variation of these families of non-intersecting lattice paths so that each object actually gets assigned a monomial in $u,v,w,X_1,\ldots,X_n$ (and the exponents can be viewed as statistics). For that we only need to consider a variation that concerns the vertical edges strictly left of the $y$-axis (since their weights are not monomials): We simply duplicate them and assign to one of them the weight $X_m v^{-1} w$ (where its top vertex $(p,q)$ satisfies $p+q+2=m$), while the other gets assigned $1$. 

\subsection{From non-intersecting lattice paths to plane partitions}
Finally, we work out how these families of non-intersecting lattice paths correspond to pairs of a CSSPP $R$ and a row strict shifted plane partition (RSSPP) $L$, where the latter is defined as a filling of a shifted Young diagram with positive integers that decrease strictly along rows and weakly down columns: The CSSPP $R$ corresponds to the portions of the paths that are right of the $y$-axis and the RSSPP $L$
 to the portions of the paths that are left of the $y$-axis. The entries correspond to the horizontal steps, and are essentially determined by their heights as described in the next paragraph. The diagonal steps indicate that respective row lengths of the CSSPP and the RSSPP  
differ by $1$. The pair of shifted PPs that correspond to the extended DPP paths from Figure~\ref{nilp2} is displayed next.
$$
L= \begin{ytableau}
8 & 7 & 6 & 5 & 3 & 2 & 1 \\
\none & 7 & 6 & 4 & 2 & 1 \\
\none & \none & 5 & 3 & 2  \\
\none & \none & \none & 3 & 1
\end{ytableau} \qquad
R= \begin{ytableau}
8 & 8 & 7 & 7 & 6 & 4 & 1 \\
\none & 7 & 6 & 5 & 5  \\
\none & \none & 4 & 4 & 4  \\
\none & \none & \none & 3 & 2
\end{ytableau}
$$
More precisely, $R$ 
 collects the heights of the horizontal steps increased by $2$ of the portions of the paths that are located right of the $y$-axis, while $L$ 
  is determined by the horizontal steps of the portions of the paths left of the $y$-axis as follows: again add $2$ to these heights, subtract their distances (of their rightmost points) from the $y$-axis, and read these numbers from right to left.

In general, we have a pair of an RSSPP $L$ and a CSSPP $R$ with parts in $\{1,2,\ldots,n\}$ such that the following conditions are fulfilled:
\begin{enumerate}
\item The number of parts in row $i$ of $L$ is equal to or one greater than the number of parts in row $i$ of $R$.
\item The first part of row $i$ of $L$ is greater than or equal to the first part of the correspond row of $R$.
\item The first part of row $i$ of $R$ is greater than the first part of row $i+1$ of $L$.
\end{enumerate}
The weight of $R$ is as follows:
\begin{equation} 
\W_{r}( R ) = w^{\# \text{ of } i \text{ with } \ell(R_i) = \ell(L_i)-1}  \prod_{i=1}^n (u X_i)^{\# \text{ of $i$ in } R},   
\end{equation} 
where  $\ell(L_i), \ell(R_i)$ denote the lengths of rows $i$ in $L$ and $R$, respectively. As for the weight of $L$, let $L_i$ denote the set of elements in the $i$-th row of $L$, and $L^c_i$ its complement in $\{1,\ldots,\max(L_i)\}$. 
Then we set 
\begin{equation} 
\W_{l}(L) = \prod_{i} \prod_{j \in L_i} X_j v^{-1} \prod_{j \in L^c_i} (X_j v^{-1} w + 1).
\end{equation} 

Again, we would prefer monomials as weights and that can actually be achieved by using set-valued RSSPP: a \emph{set-valued RSSPP} is a filling of a shifted Young diagram with non-empty sets of positive integers such that 
\begin{itemize} 
\item there is a strict decrease along rows in the sense that, for two adjacent cells in a row, all elements in the left cell are greater than all elements in the right cell, and 
\item there is a weak decrease down columns in the sense that the maximal elements of the sets decrease.
\end{itemize}  

\begin{definition} 
\label{DPPshape}
An ($n$-)DPP pair is a pair of a set-valued RSSPP $L$ and a CSSPP $R$ with parts in $\{1,2,\ldots,n\}$ such that the following conditions are fulfilled:
\begin{enumerate}
\item The number of cells in row $i$ of $L$ is equal to or one greater than the number of cells in row $i$ of $R$.
\item The maximum part of row $i$ of $L$ is greater than or equal to the maximum part of row $i$ of $R$.
\item The maximum part of row $i$ of $R$ is greater than the maximum part of row $i+1$ of $L$.
\end{enumerate}
The weight of $R$ is as follows:
\begin{equation} 
\W_{r}( R ) = w^{\# \text{ of } i \text{ with } \ell(R_i) = \ell(L_i)-1}  \prod_{i=1}^n (u X_i)^{\# \text{ of $i$ in } R},   
\end{equation} 
where  $\ell(L_i), \ell(R_i)$ denote the lengths of rows $i$ in $L$ and $R$, respectively. As for the weight of $L$, let $L_i$ denote the set of elements in the $i$-th row of $L$, then we set 
\begin{equation} 
\W_{l}(L) = \prod_{i} w^{|L_i| - 
\ell(L_i)}   \prod_{j \in L_i} X_j v^{-1}.
\end{equation} 
The total weight of the DPP pair $(L,R)$ is then 
\begin{equation} 
\W(L,R) = v^{\binom{n+1}{2}} \cdot \W_{l}(L) \cdot \W_{r} ( R ).
\end{equation} 
\end{definition}

Alternatively, we can combine an RSSPP $L$ and an CSSPP $R$ satisfying the conditions as described in the beginning of this subsection 
 into a single column strict plane partition: reflecting $L$ along the main diagonal and merge the reflected $L$ and $R$ into a single CSPP such that the first cell of row $i$ of $L$ (before reflecting) is left of the first cell of row $i$ of $R$ and below the first cell of row $i-1$ of $R$. In our example we obtain (with the entries in $L$ underlayed with blue) the following CSPP.
\begin{center}
		\begin{tikzpicture}[scale=0.7]
		\draw[fill=cyan!25!white] (0,0) -- (1,0) -- (1,-1) -- (2,-1) -- (2,-2) -- (3,-2) -- (3,-3) -- (4,-3) -- (4,-5) -- (2,-5) -- (2,-6) -- (1,-6) -- (1,-7) -- (0,-7) -- (0,0);
		\foreach \x in {1,...,8}
			\draw (\x-1,0) rectangle (\x,-1);
		\foreach \x in {1,...,6}
			\draw (\x-1,-1) rectangle (\x,-2);
		\foreach \x in {1,...,6}
			\draw (\x-1,-2) rectangle (\x,-3);
		\foreach \x in {1,...,6}
			\draw (\x-1,-3) rectangle (\x,-4);
		\foreach \x in {1,...,4}
			\draw (\x-1,-4) rectangle (\x,-5);
		\foreach \x in {1,...,2}
			\draw (\x-1,-5) rectangle (\x,-6);
		\foreach \x in {1,...,1}
			\draw (\x-1,-6) rectangle (\x,-7);
		\node at ( .5,-.5) {$8$};
		\node at (1.5,-.5) {$8$};
		\node at (2.5,-.5) {$8$};
		\node at (3.5,-.5) {$7$};
		\node at (4.5,-.5) {$7$};
		\node at (5.5,-.5) {$6$};
		\node at (6.5,-.5) {$4$};
		\node at (7.5,-.5) {$1$};
		\node at ( .5,-1.5) {$7$};
		\node at (1.5,-1.5) {$7$};
		\node at (2.5,-1.5) {$7$};
		\node at (3.5,-1.5) {$6$};
		\node at (4.5,-1.5) {$5$};
		\node at (5.5,-1.5) {$5$};
		\node at ( .5,-2.5) {$6$};
		\node at (1.5,-2.5) {$6$};
		\node at (2.5,-2.5) {$5$};
		\node at (3.5,-2.5) {$4$};
		\node at (4.5,-2.5) {$4$};
		\node at (5.5,-2.5) {$4$};
		\node at ( .5,-3.5) {$5$};
		\node at (1.5,-3.5) {$4$};
		\node at (2.5,-3.5) {$3$};
		\node at (3.5,-3.5) {$3$};
		\node at (4.5,-3.5) {$3$};
		\node at (5.5,-3.5) {$2$};	
		\node at ( .5,-4.5) {$3$};
		\node at (1.5,-4.5) {$2$};
		\node at (2.5,-4.5) {$2$};
		\node at (3.5,-4.5) {$1$};
		\node at ( .5,-5.5) {$2$};
		\node at (1.5,-5.5) {$1$};
		\node at ( .5,-6.5) {$1$};	
		\end{tikzpicture}
\end{center}		

\begin{definition}
\label{BCSPP}
 A BCSPP $D=(d_{i,j})$ of order $n$ is a CSPP with parts in $\{1,2,\ldots,n\}$ whose shape $\lambda$ is near-balanced. 

To define the weight, let for each column $j$, $S_j$ denote the integers in $\{1,2,\ldots,d_{j,j}\}$ that are \emph{not} in column $j$ below the main diagonal.  Then the weight is defined as follows.
\begin{multline} 
\W(D) = \W(\lambda) \cdot u^{\# \text{of cells strictly above the main diagonal}} \cdot v^{\binom{n+1}{2} -\# \text{of cells on and below the main diagonal}} \\
\times \prod_{i,j}^n X_{d_{i,j}}  \prod_{j} \prod_{e \in S_j} (X_e v^{-1} w +1) 
\end{multline} 
\end{definition} 

Again, it is more desirable to have objects whose weights are monomials in $u,v,w,X_1,\ldots,X_n$, and this finally led to Definition~\ref{PP}.

\section{Proof of Theorem~\ref{robbins}}
\label{constantterm}

This section is devoted to the proof of Theorem~\ref{robbins}. The proof is by induction with respect to $n$ and uses the recursion ``underlying'' arrowed monotone triangles. More specifically, observe that the deletion of the bottom row of an arrowed monotone triangle leads to an arrowed monotone triangle with one row less. The recursion follows when fixing the bottom row of the bigger arrowed monotone triangle and reading off the conditions for the penultimate row.

For this inductive proof, it is best to work with the Jacobi-Trudi formula for the extended Schur polynomial, and, therefore, we start by
providing the Jacobi-Trudi-type determinant for the extended Schur polynomials $s_{(k_n,k_{n-1},\ldots,k_1)}(X_1,\ldots,X_n)$. 
We define 
an extension of \emph{complete homogeneous symmetric functions} to negative parameters: As usual we have for non-negative integers  $k$  
$$
h_k(X_1,\ldots,X_n) = \sum_{l_1+\dots+l_n=k, l_i \ge 0} X_1^{l_1} X_2^{l_2} \cdots X_n^{l_n},  
$$
while for negative $k$, we define 
$$
h_k(X_1,\ldots,X_n) = (-1)^{n+1} \sum_{l_1+\dots+l_n=k, l_i < 0} X_1^{l_1} X_2^{l_2} \cdots X_n^{l_n}.
$$
In particular, it follows that $h_k(X_1,\ldots,X_n)=0$ if $-n < k < 0$, since then the range of the summation is empty. We then have the following reciprocity for all $k$.
$$
h_k(X_1,\ldots,X_n) =(-1)^{n+1}X_1^{-1}\cdots X_n^{-1}h_{-k-n}(X_1^{-1},\ldots,X_n^{-1})
$$
It can be checked\footnote{This can be done using Lemma~\ref{dettodet}. In fact, one obtains the determinant \eqref{i} that way, but by the indicated column operations the two determinants are equal.
} 
that the Jacobi-Trudi-type determinant for the extended Schur polynomials reads as follows:
\begin{equation} 
s_{(k_n,\ldots,k_1)}(X_1,\ldots,X_n) = \det_{1 \le i, j \le n} \left( h_{k_i+i-j}(X_1,\ldots,X_n) \right)
\end{equation} 
Applying elementary column operations based on the identity 
\begin{equation}
\label{hrec} 
h_k(X_1,\ldots,X_n) = h_k(X_1,\ldots,\widehat{X_i},\ldots,X_n) + X_i \cdot h_{k-1}(X_1,\ldots,X_n),
\end{equation} 
where $\widehat{X_i}$ indicates that $X_i$ is missing,
it is easy to see that also  
\begin{equation} 
\label{i} 
s_{(k_n,\ldots,k_1)}(X_1,\ldots,X_n) = \det_{1 \le i, j \le n} \left( h_{k_i+i-j}(X_{n-j+1},\ldots,X_n) \right).
\end{equation} 
A combinatorial interpretation of $s_{(k_n,\ldots,k_1)}(X_1,\ldots,X_n)$ as generating function is provided by the subset of arrowed monotone triangle with bottom row $k_1,\ldots,k_n$ with all entries being decorated with $\emptyset$. This corresponds to setting 
\begin{equation} 
\label{special}
\wne =  \wnw = \wnenw = 0   \qquad \text{and} \qquad \wnone=1.
\end{equation}
In the special case $0 \le k_1 \le k_2 \le \ldots \le k_n$, these objects are known as Gelfand-Tsetlin patterns (which are in easy bijection with semistandard tableaux), however, the important thing to note is that we now have an interpretation as signed objects for all $(k_1,\ldots,k_n) \in \mathbb{Z}^n$. The fact that the generating function is given by the extended Schur polynomial is a consequence of Theorem~\ref{robbins}.

Let $\alpha(k_1,\ldots,k_n;X_1,\ldots,X_n)$ denote the generating function of arrowed monotone triangles with bottom row $k_1,\ldots,k_n$, as considered in Theorem~\ref{robbins}.
The proof is by induction with respect to $n$. We have 
$$
\alpha(k_1;X_1) = (\wne \cdot X_1 + \wnw \cdot X_1^{-1} + \wnenw + \wnone) \cdot X_1^{k_1},
$$
and this coincides with the formula. We set 
\begin{equation} 
V_{p,q,n}=  \wne \cdot X_n \e_{q} + \wnw \cdot X_n^{-1} \e_{p}^{-1}  + 
\wnenw \cdot  \e_{p}^{-1} \e_{q} + \wnone \cdot \id
\end{equation} 
and have the following recursion, which is immediate from the definition.
\begin{multline}
\label{recursion}
\alpha(k_1,\ldots,k_n;X_1,\ldots,X_n) = X_n^{k_1+\ldots+k_n} \\
\times \left[ V_{k_1,k'_1,n} V_{k_2,k'_2,n} \cdots V_{k_n,k'_n,n} \sum_{l_1=k'_1}^{k_2} \sum_{l_2=k'_2}^{k_3} \ldots \sum_{l_{n-1}=k'_{n-1}}^{k_n}
X_n^{-l_1-l_2-\ldots-l_{n-1}} \alpha(l_1,\ldots,l_{n-1};X_1,\ldots,X_{n-1}) \right]_{k_i=k_i'}
\end{multline} 
Indeed, if we delete the bottom row of an arrowed monotone triangle, we obtain an arrowed monotone triangle with one row less. Moreover, 
$\wne \cdot X_n \e_{k_i'}$ in $V_{k_i,k'_i,n}$ corresponds to decorating $k_i$ with $\nearrow$, 
$\wnw \cdot X_n^{-1} \e_{k_i}^{-1}$ corresponds to decorating $k_i$ with $\nwarrow$, 
$\wnenw \cdot  \e_{k_i}^{-1} \e_{k'_i}$ corresponds to decorating $k_i$ with $\nenwarrow$, and 
$\wnone \cdot \id$ corresponds to decorating $k_i$ with $\emptyset$.

Here and in the following $\sum\limits_{i=a}^{b} f(i)$ has to be interpreted as $\sum\limits_{i \in \si{a}{b}} f(i)$.
Using the induction hypothesis, the right-hand side is equal to 
\begin{multline*}
X_n^{k_1+\ldots+k_n} \left[ V_{k_1,k'_1,n} V_{k_2,k'_2,n} \cdots V_{k_n,k'_n,n} \sum_{l_1=k'_1}^{k_2} \sum_{l_2=k'_2}^{k_3} \ldots \sum_{l_{n-1}=k'_{n-1}}^{k_n}  
X_n^{-l_1-l_2-\ldots-l_{n-1}} \right. \\
\times \prod_{i=1}^{n-1} (\wne \cdot X_i + \wnw \cdot X_i^{-1} + \wnenw + \wnone) \\ 
\prod_{1 \le p < q \le n-1} \left( \wne \cdot \e_{l_p} + \wnw \cdot \e_{l_q}^{-1} + \wnenw \cdot \e_{l_p} \e_{l_q}^{-1} + \wnone \cdot \id  \right) \\  \left. 
\det_{1 \le i, j \le n-1} \left( h_{l_i+i-j}(X_{n-j},\ldots,X_{n-1}) \right) \right]_{k_i=k'_i}. 
\end{multline*}
We define the following difference operator $\Delta_{k,X}$:
\begin{equation} 
\Delta_{k,X} p(k) := p(k) - X \cdot p(k-1)
\end{equation} 
Note that 
\begin{equation}
\label{diff}
\Delta_{k,X_i} h_k(X_1,\ldots,X_n) = h_k(X_1,\ldots,X_n) - X_i \cdot h_{k-1}(X_1,\ldots,X_n) = 
h_k(X_1,\ldots,\widehat{X_i},\ldots,X_n)
\end{equation}
and 
\begin{equation}
\label{sum}
\sum_{l_i=k_i}^{k_{i+1}}  X_n^{-l_i} \Delta_{l_i,X_n} p(l_i) = \sum_{l_i=k_i}^{k_{i+1}}  \left( X_n^{-l_i} p(l_i) - X_n^{-l_i+1} p(l_i-1) \right) = X_n^{-k_{i+1}} p(k_{i+1}) - X_n^{-k_i+1} p(k_i-1).
\end{equation}
Using 
$$
\det_{1 \le i, j \le n-1} \left( h_{l_i+i-j}(X_{n-j},\ldots,X_{n-1}) \right) = 
\Delta_{l_1,X_n} \Delta_{l_2,X_n} \cdots \Delta_{l_{n-1} ,X_n} \det_{1 \le i, j \le n-1}  \left( h_{l_i+i-j}(X_{n-j},\ldots,X_{n}) \right),
$$
we obtain 
\begin{multline*}
X_n^{k_1+\ldots+k_n} \prod_{i=1}^{n-1} (\wne \cdot X_i + \wnw \cdot X_i^{-1} + \wnenw + \wnone) \\ 
\times \left[ V_{k_1,k'_1,n} V_{k_2,k'_2,n} \cdots V_{k_n,k'_n,n} \sum_{l_1=k'_1}^{k_2} \sum_{l_2=k'_2}^{k_3} \ldots \sum_{l_{n-1}=k'_{n-1}}^{k_n}  X_n^{-l_1-l_2-\ldots-l_{n-1}} \right. \\ 
\left. \phantom{\sum_{i=a}^{b}}  \Delta_{l_1,X_n} \Delta_{l_2,X_n} \cdots \Delta_{l_{n-1} ,X_n}  \widehat{\alpha}(l_1,\ldots,l_{n-1};X_1,\ldots, X_{n}) \right]_{k_i=k'_i}
\end{multline*}
with 
\begin{multline*}
\widehat{\alpha}(l_1,\ldots,l_{n-1};X_1,\ldots, X_{n})  
= \prod_{1 \le p < q \le n-1} \left( \wne \cdot \e_{l_p} + \wnw \cdot \e_{l_q}^{-1} + \wnenw \cdot \e_{l_p} \e_{l_q}^{-1} + \wnone \cdot \id  \right) \\  \det_{1 \le i, j \le n-1} \left( h_{l_i+i-j}(X_{n-j},\ldots,X_{n}) \right).
\end{multline*}
We now apply \eqref{sum} and obtain the following. 
\begin{multline}
\label{zwischen}
X_n^{k_1+\ldots+k_n} \prod_{i=1}^{n-1} (\wne \cdot X_i + \wnw \cdot X_i^{-1} + \wnenw + \wnone) \\ 
\times \left[ V_{k_1,k'_1,n} V_{k_2,k'_2,n} \cdots V_{k_n,k'_n,n} \sum_{(l_1,\ldots,l_{n-1}): \atop l_{i}=k'_i-1 \text{ or } l_{i}=k_{i+1}} (-1)^{\# i: \, l_i = k'_i-1} X_n^{-l_1-l_2-\ldots-l_{n-1}}  \widehat{\alpha}(l_1,\ldots,l_{n-1};X_1,\ldots, X_{n}) \right]_{k_i=k'_i}
\end{multline}
The sum over the $(l_1,\ldots,l_{n-1})$ has $2^{n-1}$ terms. Next we argue that the summands with $l_{i-1}=k_i, l_i=k'_i-1$ for at least one $i$ evaluate to zero after applying $V_{k_i,k'_i,n}$ and setting $k_i=k'_i$ afterwards. We set 
$$
\widehat{\beta}(l_1,\ldots,l_{n-1};X_1,\ldots, X_{n}) = \det_{1 \le i, j \le n-1} \left( h_{l_i+i-j}(X_{n-j},\ldots,X_{n}) \right), 
$$
and using $X^{-l} \e^{\pm 1}_l a(l)=X^{-l}  a(l \pm 1) = X^{\pm 1} \e^{\pm 1}_l X^{-l} a(l)$, it follows that  
\begin{multline*} 
X_n^{-l_1-l_2-\ldots-l_{n-1}}  \widehat{\alpha}(l_1,\ldots,l_{n-1};X_1,\ldots, X_{n-1}) \\
= \prod_{1 \le p < q \le n-1} \left( \wne \cdot  X_n \e_{l_p} + \wnw \cdot X_n^{-1} \e_{l_q}^{-1} + \wnenw \cdot \e_{l_p} \e_{l_q}^{-1} + \wnone \cdot \id  \right) \\ 
X_n^{-l_1-l_2-\ldots-l_{n-1}} \widehat{\beta}(l_1,\ldots,l_{n-1};X_1,\ldots, X_{n}) \\
= \prod_{1 \le p < q \le n-1} V_{l_q,l_p,n}  X_n^{-l_1-l_2-\ldots-l_{n-1}} \widehat{\beta}(l_1,\ldots,l_{n-1};X_1,\ldots, X_{n}).
\end{multline*} 
We consider a summand in \eqref{zwischen} with $l_{i-1}=k_i$ and $l_i=k'_i-1$.
Up to some unimportant factors and specializations (we omit the first line of \eqref{zwischen}, $V_{k_j,k_j'}$ for $j \not= i$, the sign 
$(-1)^{\# j: \, l_j = k'_j-1}$ as well as the specialization $k_j=k'_j$ for $j \not=i$), such a summand in \eqref{zwischen} can also be written as 
\begin{multline*}
\left[ V_{l_{i-1},l_i,n} \prod_{1 \le p < q \le n-1} V_{l_q,l_p,n}  X_n^{-l_1-l_2-\ldots-l_{n-1}} \widehat{\beta}(l_1,\ldots,l_{n-1};X_1,\ldots, X_{n})  \right]_{l_{i-1}=k_i,l_i=k_i-1}  \\
= \left[ X_n V_{l_{i-1},l_i,n} V_{l_i,l_{i-1},n} \prod_{1 \le p < q \le n-1 \atop (p,q) \not= (i-1,i)} V_{l_q,l_p,n}  X_n^{-l_1-l_2-\ldots-l_{n-1}} \e_{l_i}^{-1} \widehat{\beta}(l_1,\ldots,l_{n-1};X_1,\ldots, X_{n}) \right]_{l_{i-1}=l_i=k_i}.
\end{multline*}
(In the first line, the crucial observation is that we can replace $V_{k_i,k_i',n}$ from \eqref{zwischen} with $V_{l_{i-1},l_i,n}$, taking $l_{i-1}=k_i$ and $l_{i}=k'_{i}-1$ into account and the specialization.)
The summand vanishes since $\e_{l_i}^{-1} \widehat{\beta}(l_1,\ldots,l_{n-1};X_1,\ldots, X_{n})$ is asymmetric in $l_{i-1},l_i$ and 
$$V_{l_{i-1},l_i,n} V_{l_i,l_{i-1},n} \prod\limits_{1 \le p < q \le n-1 \atop (p,q) \not= (i-1,i)} V_{l_q,l_p,n}$$ 
is symmetric in $l_{i-1},l_i$ as this implies that the function in brackets is also asymmetric. It follows that \eqref{zwischen} is equal to 
\begin{multline*}
X_n^{k_1+\ldots+k_n} \prod_{i=1}^{n-1} (\wne \cdot X_i + \wnw \cdot X_i^{-1} + \wnenw + \wnone) \left[ V_{k_1,k'_1,n} V_{k_2,k'_2,n} \cdots V_{k_n,k'_n,n}  \phantom{\sum_{i=a}^b} \right. \\ \left. \sum_{i=1}^n (-1)^{i-1} X_n^{-k'_1-\ldots-k'_{i-1}-k_{i+1}-\ldots-k_n+i-1}  \widehat{\alpha}(
k_1^\prime-1,\ldots,k_{i-1}^\prime-1,k_{i+1},\ldots,k_n;X_1,\ldots, X_{n}) \right]_{k'_j=k_j}.
\end{multline*}
In this expression, the $i$-th summand is independent of $k_i,k_i'$, and the action of $V_{k_i,k'_i}$ is then simply multiplication with 
$\wne \cdot X_n + \wnw \cdot X_n^{-1} + \wnenw + \wnone$, which is due to its independence of $i$ an overall factor and
fits as last factor within the product in the front. Moreover, 
the $i$-th summand is independent of $k_j$ for $j<i$, and so we may replace $V_{k_j,k'_j}$ by $V_{k_i,k_j'}$ (since the summand is also independent of $k_i$), and for a similar reason we can replace $V_{k_j,k'_j}$ by $V_{k_j,k_i}$ for $j>i$.
We can specialize $k_j=k'_j$ now, and obtain
\begin{multline*} 
 X_n^{k_1+\ldots+k_n} \prod_{i=1}^{n} (\wne \cdot X_i + \wnw \cdot X_i^{-1} + \wnenw + \wnone)  
 \\ \times  \sum_{i=1}^n (-1)^{i-1} \prod_{j=1}^{i-1} V_{k_i,k_j} \prod_{j=i+1}^{n} V_{k_j,k_i}  X_n^{-k_1-\ldots-k_{i-1}-k_{i+1}-\ldots-k_n+i-1}  \\
 \times \widehat{\alpha}(k_1-1,\ldots,k_{i-1}-1,k_{i+1},\ldots,k_n;X_1,\ldots, X_{n}).
 \end{multline*} 
 We use $X^{-l} \e^{\pm 1}_l a(l)=X^{-l}  a(l \pm 1) = X^{\pm 1} \e^{\pm 1}_l X^{-l} a(l)$ again to arrive at
 \begin{multline*} 
 \prod_{i=1}^{n} (\wne \cdot X_i + \wnw \cdot X_i^{-1} + \wnenw + \wnone)  \\ 
 \prod_{1 \le p < q \le n} \left( \wne \cdot \e_{k_p} + \wnw \cdot \e_{k_q}^{-1} + \wnenw \cdot \e_{k_p} \e_{k_q}^{-1} + \wnone \cdot \id  \right) \\
 \sum_{i=1}^n (-1)^{i-1} X_n^{k_i+i-1} \widehat{\beta}(k_1-1,\ldots,k_{i-1}-1,k_{i+1},\ldots,k_{n};X_1,\ldots, X_{n}).
\end{multline*}
Now 
$$
\sum_{i=1}^n (-1)^{i-1} X_n^{k_i+i-1} \widehat{\beta}(k_1-1,\ldots,k_{i-1}-1,k_{i+1},\ldots,k_{n};X_1,\ldots, X_{n}) = 
\det_{1 \le i, j \le n} \left( h_{k_i+i-j}(X_{n-j+1},\ldots,X_n) \right)
$$
follows from expanding the determinant with respect to the first column. This concludes the proof of Theorem~~\ref{robbins}.

\section{Further generalization of Theorem~\ref{robbins0}}
\label{decorated}

There is the following straightforward generalization of arrowed monotone triangles: A \emph{generalized extended arrowed monotone triangle} (AMT) is a triangular array of the following form
$$
\begin{array}{ccccccccccccccccc}
  &   &   &   &   &   &   &   & a_{1,1} &   &   &   &   &   &   &   & \\
  &   &   &   &   &   &   & a_{2,1} &   & a_{2,2} &   &   &   &   &   &   & \\
  &   &   &   &   &   & \dots &   & \dots &   & \dots &   &   &   &   &   & \\
  &   &   &   &   & a_{n-2,1} &   & \dots &   & \dots &   & a_{n-2,n-2} &   &   &   &   & \\
  &   &   &   & a_{n-1,1} &   & a_{n-1,2} &  &   \dots &   & \dots   &  & a_{n-1,n-1}  &   &   &   & \\
  &   &   & a_{n,1} &   & a_{n,2} &   & a_{n,3} &   & \dots &   & \dots &   & a_{n,n} &   &   &
\end{array},
$$
where each $a_{i,j}$ is an integer decorated with an element from $\{\nenwarrowij{s}{t} | s,t \in \mathbb{Z} \}$ and the following is satisfied for each integer $a$ not in the bottom row: Suppose $b$ is the 
southwest-neighbor and $c$ is the 
southeast-neighbor, and $\dec(b) = \nenwarrowij{s_b}{t_b}$ and $\dec(c) = \nenwarrowij{s_c}{t_c}$, then we require
$a \in \si{b+t_b}{c-s_c}$. Extended arrowed monotone triangles are the subclass of generalized extended arrowed monotone triangles 
with possible decorations $\nenwarrowij{s}{t}$ with $s,t \in \{0,1\}$. 

Suppose we are given arbitrary weights 
$\omega(\nenwarrowij{s}{t})$ for all $s,t \in \mathbb{Z}$, and, moreover, we define $\alpha(\nenwarrowij{s}{t})=t-s$.
 Then the weight of a given generalized extended arrowed monotone triangle 
$A=(a_{i,j})_{1 \le j \le i \le n}$
is 
\begin{equation}
\W(A) = \prod_{p,q \in \mathbb{Z}} \omega(\nenwarrowij{p}{q})^{\# \text{ of } \nenwarrowijs{p}{q}}
\prod_{i=1}^{n} X_i^{\sum_{j=1}^i a_{i,j}  - \sum_{j=1}^{i-1} a_{i-1,j} + \sum_{j=1}^i \alpha(\dec(a_{i,j}))}.
\end{equation}

Theorem~\ref{robbins} has the following extension. 

\begin{theorem} 
\label{robbinsgen}
Suppose the weights $\omega(\nenwarrowij{p}{q})=0$, for all but finitely many $p,q \in \mathbb{Z}$.
The generating function of generalized extended arrowed monotone triangles with bottom row $k_1,\ldots,k_n$ is 
\begin{equation}
\prod_{i=1}^{n} \sum_{s,t \in \mathbb{Z}} \omega(\nenwarrowij{s}{t}) X_i^{t-s}  \\ 
\prod_{1 \le p < q \le n} \left( \sum_{s,t \in \mathbb{Z}} \omega(\nenwarrowij{s}{t}) \e_{k_p}^t \e_{k_q}^{-s} \right)
s_{(k_n,k_{n-1},\ldots,k_1)}(X_1,\ldots,X_n). 
\end{equation}
\end{theorem}

The proof of this theorem is analogous to the proof of Theorem~\ref{robbins} and thus omitted in this paper.

\section{From the case $(k_1,\ldots,k_n)=(1,2,\ldots,n)$ to DPPs} 
\label{proofmain2}

The purpose of this section is to complete the proof of Theorem~\ref{main2}. The main ingredient is the generating function from Theorem~\ref{main1} for the special case $(k_1,\ldots,k_n)=(1,2,\ldots,n)$. Roughly speaking, we apply several algebraic manipulations until we arrive at a determinant that we can interpret combinatorially as a weighted count of non-intersecting lattice paths that we have encountered in Section~\ref{extendedpaths}. These manipulations are structured as follows.
\begin{itemize} 
\item In Section~\ref{optoasym}, we transform \eqref{opX} from Theorem~\ref{main1} for the special case $(k_1,\ldots,k_n)=(1,2,\ldots,n)$ into an antisymmetrizer expression.
\item In Section~\ref{asymtobia}, we use Lemma~\ref{general} to transform the antisymmetrizer expression into a bialternant.
\item In Section~\ref{biatojacobitrudi}, we use Lemma~\ref{dettodet} to transform the bialternant into a Jacobi-Trudi-type determinant.
\item In Section~\ref{further}, we multiply the matrix underlying the determinant with a matrix that has an easy-to-compute determinant which just cancels the prefactor to arrive at the final determinant.
\item In Section~\ref{GVlast}, we interpret this final determinant combinatorally using the Lindstr\"om-Gessel-Viennot lemma and arrive at a weighted count of the non-intersecting lattice paths from Section~\ref{extendedpaths}.
\end{itemize}

\subsection{From the operator formula to an antisymmetrizer expression}
\label{optoasym} 
Using the antisymmetrizer, defined as 
\begin{equation} 
\asym_{X_1,\ldots,X_n} \left[ f(X_{1},\ldots,X_{n}) \right] = \sum_{\sigma \in {\mathcal S}_n}  \sgn \sigma 
f(X_{\sigma(1)},\ldots,X_{\sigma(n)}), 
\end{equation} 
we rewrite the expression from Theorem~\ref{robbins} for $\wnone=0, \wne = u, \wnw = v, \wnenw = w$ and obtain
\begin{multline}
\prod_{i=1}^{n} (u  X_i + v  X_i^{-1}+w) \prod_{1 \le p < q \le n} \left( u  \e_{k_p} + v  \e_{k_q}^{-1} + w  \e_{k_p} \e_{k_q}^{-1} \right) 
s_{(k_n,k_{n-1},\ldots,k_1)}(X_1,\ldots,X_n) \\
= \prod_{i=1}^{n} (u  X_i + v  X_i^{-1}+w) \prod_{1 \le p < q \le n} \left( u  \e_{k_p} + v  \e_{k_q}^{-1} + w  \e_{k_p} \e_{k_q}^{-1} \right)
\frac{\asym_{X_1,\ldots,X_n} \left[ \prod_{i=1}^{n} X_i^{k_i+i-1} \right] }{\prod_{1 \le i < j \le n} (X_j - X_i)} \\ 
= \prod_{i=1}^{n} (u  X_i + v  X_i^{-1}+w)  \frac{\asym_{X_1,\ldots,X_n} \left[ \prod_{1 \le p < q \le n} \left( u  \e_{k_p} + v  \e_{k_q}^{-1} + w  \e_{k_p} \e_{k_q}^{-1} \right) \prod_{i=1}^{n} X_i^{k_i+i-1} \right] }{\prod_{1 \le i < j \le n} (X_j - X_i)} \\ 
= \prod_{i=1}^{n} (u  X_i + v  X_i^{-1}+w) \frac{\asym_{X_1,\ldots,X_n} \left[ \prod_{1 \le p <  q \le n} \left( u X_p  + v  X_q^{-1} + w  X_p X_q^{-1} \right) \prod_{i=1}^{n} X_i^{k_i+i-1} \right] }{\prod_{1 \le i < j \le n} (X_j - X_i)} \\
= \prod_{i=1}^{n} (u  X_i + v  X_i^{-1}+w) \frac{\asym_{X_1,\ldots,X_n} \left[ \prod_{1 \le p < q \le n} \left( u X_p X_q + v   + w  X_p  \right) \prod_{i=1}^{n} X_i^{k_i} \right] }{\prod_{1 \le i < j \le n} (X_j - X_i)}  \\
=  \frac{\asym_{X_1,\ldots,X_n} \left[ \prod_{1 \le p \le q \le n} \left( u  X_q + v  X_p^{-1}  + w  \right) \prod_{i=1}^{n} X_i^{k_i+n-i} \right] }{\prod_{1 \le i < j \le n} (X_j - X_i)}.
\end{multline} 

From now on, we will consider the special case $(k_1,k_2,\ldots,k_n) = (1,2,\ldots,n)$, and therefore, we have the following expression.
\begin{equation} 
\prod_{i=1}^n X_i^{n} \frac{\asym_{X_1,\ldots,X_n} \left[ \prod_{1 \le p \le q \le n} \left( u  X_q + v  X_p^{-1}  + w  \right)  \right] }{\prod_{1 \le i < j \le n} (X_j - X_i)}
\end{equation} 

\subsection{From an antisymmetrizer
 to a bialternant}
\label{asymtobia}
We need the following lemma to obtain a determinantal expression.

\begin{lemma}  
\label{general}
Let $n \ge 1$, and $\mathbb{X}=(X_1,\ldots,X_n), \mathbb{Y}=(Y_1,\ldots,Y_n)$ be algebraically independent indeterminants.
 Then
\begin{equation} 
\det_{1 \le i, j \le n} \left( X_i^j - Y_i^j \right) 
 = \widehat{\asym}  \left[  \prod_{1 \le i \le j \le n} (X_j-Y_i)     \right],
\end{equation}
with 
\begin{equation} 
\widehat{\asym} \left[f(\mathbb{X};\mathbb{Y})\right] = \sum_{\sigma \in {\mathcal{S}_n}} \sgn \sigma f(X_{\sigma(1)},\ldots,X_{\sigma(n)};Y_{\sigma(1)},\ldots,Y_{\sigma(n)}).
\end{equation} 
\end{lemma}

\begin{proof}  The proof is by induction with respect to $n$. The result is obvious for $n=1$. 
Let $L_n(\mathbb{X};\mathbb{Y}), R_n(\mathbb{X};\mathbb{Y})$ denote the left and right-hand side of the identity in the statement, respectively. By the induction hypothesis, we can assume 
$ L_{n-1}(X_1,\ldots,X_{n-1};Y_1,\ldots,Y_{n-1}) \allowbreak = R_{n-1}(X_1,\ldots,X_{n-1};Y_1,\ldots,Y_{n-1})$.
We show that both $L_n(\mathbb{X};\mathbb{Y})$ and $R_n(\mathbb{X};\mathbb{Y})$ can be computed recursively using 
$L_{n-1}(X_1,\ldots,X_{n-1};Y_1,\ldots,Y_{n-1})$ and $R_{n-1}(X_1,\ldots,X_{n-1};Y_1,\ldots,Y_{n-1})$, respectively, with the same recursion.
For the right-hand side, we have 
$$
R_n(\mathbb{X};\mathbb{Y})
=  \sum_{i=1}^{n} (-1)^{i+1}  \left( \prod_{k=1}^n (X_k-Y_i)  \right) 
R_{n-1}(X_1,\ldots,\widehat{X_i},\ldots,X_n;Y_1,\ldots,\widehat{Y_i},\ldots,Y_n),
$$
where $\widehat{X_i}$ and $\widehat{Y_i}$ means that $X_i$ and $Y_i$ are omitted.
In order to deal with the left-hand side, we first observe
\begin{equation}
\label{fundamentalidentity}
\sum_{j=0}^{n} (X_i^j - Y_i^j) e_{n-j}(-X_1,\ldots,-X_n) = (-1)^{n-1} \prod_{k=1}^{n} (X_k-Y_i),
\end{equation}
where $e_{j}(X_1,\ldots,X_n)$ denotes the $j$-th elementary symmetric function. Note that the summand for $j=0$ on the left-hand side is actually $0$, and can therefore be omitted.
Now consider the following system of linear equations with $n$ unknowns $c_j(\mathbb{X};\mathbb{Y})$, $1 \le j \le n$, and 
$n$ equations.
$$
\sum_{j=1}^{n} (X_i^j - Y_i^j)  c_j(\mathbb{X};\mathbb{Y})
= (-1)^{n-1} \prod_{k=1}^{n} (X_k-Y_i), \quad 1 \le i \le n
$$
The determinant of this system of equations is obviously $L_n(\mathbb{X};\mathbb{Y})$, which is non-zero as the rows are linearly independent. By \eqref{fundamentalidentity}, we know that the unique solution of this system is given by 
$
c_j(\mathbb{X};\mathbb{Y}) = e_{n-j}(-X_1,\ldots,-X_n).
$
On the other hand, by Cramer's rule, 
$$
c_n(\mathbb{X};\mathbb{Y}) = \frac{\det \limits_{1 \le i, j \le n} \left( \begin{cases}  X_i^j - Y_i^j, & \text{if $j<n$} \\
  (-1)^{n-1} \prod\limits_{k=1}^{n} (X_k-Y_i) , & \text{if $j=n$}  \end{cases} \right)}{L_n(\mathbb{X};\mathbb{Y})}.
$$
The assertion now follows from $c_n(\mathbb{X};\mathbb{Y}) = e_{0}(-X_1,\ldots,-X_n)=1$ and expanding the determinant in the numerator with respect to the last column.
\end{proof}

Letting $Y_i$ from the lemma be $- v X_i^{-1}$ and $X_j$ be $u X_j + w$, we obtain 
\begin{equation} 
\label{bialternantlast} 
\prod_{i=1}^{n} X_i^{n} \frac{\det_{1 \le i, j \le n} \left( (u X_i + w)^{j} 
- (- v X_i^{-1})^j \right)}{\prod_{1 \le i < j \le n} (X_j - X_i)}.
\end{equation} 

\subsection{From a bialternant to a Jacobi-Trudi-type determinant}
\label{biatojacobitrudi}

We aim at deriving a Jacobi-Trudi-type determinant. In \eqref{bialternantlast}, we add $t$-times the $(n-i)$-th column to the $(n-i+1)$-st column, for 
$i=1,\ldots,n-1$, in this order. We repeat this for $i=1,\ldots,n-2$, then for $i=1,2,\ldots,n-3$, and so on. We obtain
\begin{equation} 
\label{lr}
\prod_{i=1}^n X_i^n \frac{\det_{1 \le i, j \le n} \left( (u X_i + w+t)^{j-1} (u X_i + w) + 
  (- v X_i^{-1}+t)^{j-1} v X_i^{-1}  \right)}{\prod_{1 \le i < j \le n} (X_j - X_i)}, 
\end{equation} 
and set $t=-w$. 

The general procedure we use to obtain a Jacobi-Trudi-type determinant is provided in the following lemma. It is related to \cite[Eq. (43)]{BehDifZin12}.

\begin{lemma} 
\label{dettodet}
Let $f_j(X)$ be formal Laurent series for $1 \le j \le n$, and define 
\begin{equation} 
f_j[X_1,\ldots,X_i]=\sum_{k=1}^i \frac{f_j(X_k)}{\prod_{1 \le l \le i, l \not= k} (X_k - X_l)}.
\end{equation} 
Then 
\begin{equation} 
\frac{\det_{1 \le i, j \le n} \left( f_j(X_i) \right) }{\prod_{1 \le i < j \le n} (X_j - X_i)} = \det_{1 \le i, j \le n} \left( f_j[X_1,\ldots,X_i] \right).
\end{equation} 
Moreover, we have 
\begin{equation} 
f_j[X_1,\ldots,X_i] = \sum_{k \in \mathbb{Z}} \left\{ \langle X^{i+k-1} \rangle f_j(X) \right\} \cdot h_k(X_1,\ldots,X_i),
\end{equation} 
where $\langle X^{i+k-1} \rangle f_j(X)$ denotes the coefficient of $X^{i+k-1}$ in $f_j(X)$.
\end{lemma}

\begin{proof} 
Let $T_n$ denote the $n \times n$ lower triangular matrix defined as follows.
$$
(T_n)_{i,j} = \begin{cases} \prod\limits_{1 \le l \le i, l \not=j} (X_j-X_l)^{-1}, & i \ge j \\
                                                                       \hspace{2cm} 0,                            & i < j \end{cases}
$$
The first assertion follows from multiplying $\left( f_j(X_i) \right)_{1 \le i,j \le n}$ from the left with $T_n$ and the multiplicativity of the determinant.

By linearity, it suffices to show the second formula for monomials, that is, we need to show 
\begin{equation}
\label{mon}
\sum_{k=1}^i \frac{X_k^m}{\prod_{1 \le l \le i, l \not= k} (X_k - X_l)} =  h_{m-i+1}(X_1,\ldots,X_i).
\end{equation} 
We first consider the case $m \ge 0$. We multiply the identity with $t^m$, sum over all $m \ge 0$, and obtain 
$$
\sum_{k=1}^i (1- t X_k)^{-1} \prod_{1 \le l \le i, l \not= k} (X_k - X_l)^{-1} = t^{i-1} \prod_{l=1}^{i} (1 
-  t X_l)^{-1}. 
$$
This is equivalent to the original identity. We divide by the right-hand side, replace $t$ by $t^{-1}$ and arrive at 
$$
\sum_{k=1}^{i} \prod_{1 \le l \le i, l \not=k} \frac{t- X_l}{X_k-X_l} = 1.
$$
The left-hand side is obviously a polynomial in $t$ of degree no greater than $i-1$ and it evaluates to $1$ for 
$t=X_1,\ldots,X_i$. Thus the polynomial is $1$. As for $m<0$, we also multiply \eqref{mon} with $t^m$ and sum over all $m<0$. Noting that 
$$
\sum_{m<0} h_{m-i+1}(X_1,\ldots,X_i) t^m = - \prod_{l=1}^i (1-t X_l)^{-1} 
$$
(expand the right-hand side as a formal power series in $t^{-1}$), the result follows in the same way. 
\end{proof}


In our case, we need to consider 
\begin{multline*} 
f_j(X)= (u X )^{j-1} (u X + w) + 
  (- v X^{-1}-w)^{j-1} v X^{-1} \\ = u^j X^j + u^{j-1} w X^{j-1} +  \sum_{k \ge 0} \binom{j-1}{k} (-1)^{j-1} v ^{k+1} X^{-k-1} w^{j-1-k}
  \end{multline*} 
(see \eqref{lr} and set $t=-w$) and therefore 
$$  
\langle X^l \rangle f_j(X) = [l=j] u^j + [l=j-1]  u^{j-1} w + \binom{j-1}{-l-1} (-1)^{j-1} v^{-l} w^{j+l}, 
$$
where we use the Iverson bracket, i.e., $[\text{statement}]=1$ if the statement is true, and  $[\text{statement}]=0$ otherwise.
Therefore, the entry $(i,j)$ of the matrix underlying the determinant we can derive using Lemma~\ref{dettodet} is 
\begin{multline*} 
\sum_{k \in \mathbb{Z}} \left([k=j-i+1] u^j + [k=j-i]  u^{j-1} w + \binom{j-1}{-k-i} (-1)^{j-1} v^{-k-i+1} w^{j+k+i-1}\right) \cdot h_k(X_1,\ldots,X_i) \\
=u^j h_{j-i+1}(X_1,\ldots,X_i) + u^{j-1} w h_{j-i}(X_1,\ldots,X_i) 
+ \sum_{l \ge 1} \binom{j-1}{l-1} (-1)^{j-1} v^{l} w^{j-l} h_{-l-i+1}(X_1,\ldots,X_i).
\end{multline*} 

\subsection{Further manipulations: matrix multiplication}
\label{further} 
We define two matrices 
\begin{equation}
A_n = \left( u^j h_{j-i+1}(X_1,\ldots,X_i) + u^{j-1} w h_{j-i}(X_1,\ldots,X_i) \right)_{1 \le i, j \le n},  
\end{equation}
and 
\begin{equation}
B_n = \left( \sum_{l \ge 1} \binom{j-1}{l-1} (-1)^{j-1} v^{l} w^{j-l} h_{-l-i+1}(X_1,\ldots,X_i) \right)_{1 \le i, j \le n}, 
\end{equation}
so that the determinant can also be written as $\det \left( A_n + B_n \right)$. 
We will compute the inverse of $B_n$ because we aim at considering 
$$\det \left( A_n + B_n \right) = \det \left( B_n \right) \det \left( B_n^{-1} A_n + I_n \right),$$
where $I_n$ denotes the identity matrix. In the end, we will interpret $\det \left( B_n^{-1} A_n + I_n \right)$ combinatorially. 
The matrix $B_n$ can obviously be decomposed as follows.
 $$
B_n =
\left( h_{-j-i+1}(X_1,\ldots,X_i) \right)_{1 \le i, j \le n} \cdot \left( \binom{j-1}{i-1} (-1)^{j-1} v^{i} w^{j-i} \right)_{1 \le i, j \le n}
$$
Note that it is not hard to see that $\det \left( B_n \right) = v^{\binom{n+1}{2}} \prod_{1 \le i \le n} X_i^{-n}$, which cancels the prefactor in \eqref{lr}: The first factor of $B_n$ can be brought into triangular form by elementary column operations using \eqref{hrec}, while the second factor is already in triangular form.

\medskip
Next we calculate the inverse of $B_n$. Using Vandermonde summation, it is easy to check that 
$$
\left[ \left( \binom{j-1}{i-1} (-1)^{j-1} v^{i} w^{j-i} \right)_{1 \le i, j \le n} \right]^{-1} = \left( \binom{j-1}{i-1} (-1)^{j-1} v^{-j} w^{j-i} \right)_{1 \le i, j \le n}.
$$
We also need the following
\begin{equation}
\label{hh} 
 \left( h_{j-i}(X_j,X_{j+1},\ldots,X_n) \right)_{1 \le i, j \le n} \cdot \left( h_{1-i-j}(X_1,\ldots,X_i) \right)_{1 \le i, j \le n}  \\
= \left( h_{1-i-j}(X_1,\ldots,X_n) \right)_{1 \le i, j \le n}, 
\end{equation}
which can be shown by setting $a=1-j-i$ and $b=n-i+1$ in the next Lemma.

\begin{lemma}
\label{lem: sum dec of h}
Let $a,b,n$ be integers with $1 \leq b \leq n$, then
\[
h_a(X_1,\ldots,X_n) = \sum_{l=0}^{b-1}h_{a-l}(X_{b-l},\ldots,X_{n}) h_{l}(X_1,\ldots,X_{b-l}).
\]
\end{lemma}
\begin{proof}
The proof is by induction with respect to $b$. The assertion is trivial for $b=1$ and we perform the step from $b$ to $b+1$.
Using \eqref{hrec} first for $h_{a-l}$, and then for $h_l$ in reverse direction, we obtain
\begin{multline*}
h_a(X_1,\ldots,X_n) = \sum_{l=0}^{b-1}h_{a-l}(X_{b-l},\ldots,X_{n}) h_{l}(X_1,\ldots,X_{b-l}) \\
 = \sum_{l=0}^{b-1}  \Big( h_{a-l}(X_{b+1-l},\ldots,X_{n}) + X_{b-l} h_{a-l-1}(X_{b-l},\ldots,X_{n}) \Big) h_{l}(X_1,\ldots,X_{b-l}) \\ 
 = \sum_{l=0}^{b} h_{a-l}(X_{b+1-l},\ldots,X_{n}) h_{l}(X_1,\ldots,X_{b-l}) + 
 \sum_{l=-1}^{b-1} X_{b-l} h_{a-l-1}(X_{b-l},\ldots,X_{n}) h_{l}(X_1,\ldots,X_{b-l})  \\
 = \sum_{l=0}^{b} h_{a-l}(X_{b+1-l},\ldots,X_{n}) \Big(  h_{l}(X_1,\ldots,X_{b-l}) + X_{b+1-l} h_{l-1}(X_1,\ldots,X_{b+1-l}) \Big)\\
  = \sum_{l=0}^{b} h_{a-l}(X_{b+1-l},\ldots,X_{n}) h_{l}(X_1,\ldots,X_{b+1-l}). \qedhere
\end{multline*}
\end{proof}

Moreover, we have  
\begin{equation} 
\label{hinverse}
\left[ \left( h_{1-i-j}(X_1,\ldots,X_n) \right)_{1 \le i, j \le n} \right]^{-1} = \left( (-1)^{i+j} e_{i+j-1}(X_1,\ldots,X_n) \right)_{1 \le i,j \le n}, 
\end{equation}
as 
\begin{multline*} 
\sum_{k=1}^n (-1)^{j+k} h_{1-i-k}(X_1,\ldots,X_n) e_{k+j-1}(X_1,\ldots,X_n)  \\ = 
\sum_{k \ge 1} (-1)^{j+k+n-1} X_1^{-1} \cdots X_n^{-1} h_{i+k-1-n}(X_1^{-1},\ldots,X_n^{-1}) e_{k+j-1}(X_1,\ldots,X_n) \\
= \sum_{k \in \mathbb{Z}} (-1)^{n} X_1^{-1} \cdots X_n^{-1} \langle t^{1-i-k} \rangle \left[ \prod_{l=1}^n (t- X_l^{-1})^{-1} \right] 
\langle t^{k+j-1} \rangle \left[ \prod_{l=1}^n (1- t X_l) \right] \\
= (-1)^{n} X_1^{-1} \cdots X_n^{-1} \langle t^{j-i} \rangle \left[ \prod_{l=1}^n (t- X_l^{-1})^{-1} (1- t X_l) \right] 
= (-1)^n X_1^{-1} \cdots X_n^{-1} \langle t^{j-i} \rangle  \left[ (-1)^n X_1 \cdots X_n  \right] = [i=j].
\end{multline*} 
The right-hand side of \eqref{hinverse} can be decomposed as follows.
\begin{multline*} 
\left( (-1)^{i+j} e_{i+j-1}(X_1,\ldots,X_n) \right)_{1 \le i,j \le n} \\ = \left( (-1)^{i} e_{i-1}(X_1,\ldots,X_{n-j}) X_{n-j+1} \right)_{1 \le i,j \le n}
\cdot \left( (-1)^{j} e_{j-1}(X_{n-i+2},\ldots,X_n) \right)_{1 \le i,j \le n}
\end{multline*} 

Combining all this, we obtain for $B_n^{-1}$
\begin{multline*} 
\left( \binom{j-1}{i-1} (-1)^{j-1} v^{-j} w^{j-i} \right)_{1 \le i, j \le n} \cdot \left[ \left( h_{1-i-j}(X_1,\ldots,X_n) \right)_{1 \le i, j \le n} \right]^{-1} \cdot \left( h_{j-i}(X_j,X_{j+1},\ldots,X_n) \right)_{1 \le i, j \le n} \\
=  \left( \binom{j-1}{i-1} (-1)^{j-1} v^{-j} w^{j-i} \right)_{1 \le i, j \le n} \cdot  \left( (-1)^{i+j} e_{i+j-1}(X_1,\ldots,X_n) \right)_{1 \le i,j \le n} \cdot 
\left( h_{j-i}(X_j,X_{j+1},\ldots,X_n) \right)_{1 \le i, j \le n}  \\
= \left( \binom{j-1}{i-1} (-1)^{j-1} v^{-j} w^{j-i} \right)_{1 \le i, j \le n} \cdot \left( (-1)^{i} e_{i-1}(X_1,\ldots,X_{n-j}) X_{n-j+1} \right)_{1 \le i,j \le n} \\
\cdot \left( (-1)^{j} e_{j-1}(X_{n-i+2},\ldots,X_n) \right)_{1 \le i,j \le n} \cdot \left( h_{j-i}(X_j,X_{j+1},\ldots,X_n) \right)_{1 \le i, j \le n} \\
=  \left( \binom{j-1}{i-1} v^{-j} w^{j-i} \right)_{1 \le i, j \le n} \cdot \left(  e_{i-1}(X_1,\ldots,X_{n-j}) X_{n-j+1} \right)_{1 \le i,j \le n} \\
\cdot \left( (-1)^{j-1} e_{j-1}(X_{n-i+2},\ldots,X_n) \right)_{1 \le i,j \le n} \cdot \left( h_{j-i}(X_j,X_{j+1},\ldots,X_n) \right)_{1 \le i, j \le n}.
\end{multline*} 

Now multiply $A_n$ 
from the left with the product of all these matrices. First we see that
\begin{multline*}
 \left( h_{j-i}(X_j,X_{j+1},\ldots,X_n) \right)_{1 \le i, j \le n} \cdot ( u^j h_{j-i+1}(X_1,\ldots,X_i) + u^{j-1} w h_{j-i}(X_1,\ldots,X_i) )_{1 \le i, j \le n} \\ 
 = ( u^j h_{j-i+1}(X_1,\ldots,X_n) + u^{j-1} w h_{j-i}(X_1,\ldots,X_n) )_{1 \le i, j \le n}, 
\end{multline*} 
which follows (in a manner very similar to \eqref{hh}) from 
Lemma~\ref{lem: sum dec of h} for $a=j-i+1$ and $b=j+1$ or $a=j-i$ and $b=j$ respectively.
}
The entry $(i,j)$ from the matrix   
\begin{multline*} 
\left( (-1)^{j-1} e_{j-1}(X_{n-i+2},\ldots,X_n) \right)_{1 \le i,j \le n} \cdot ( u^j h_{j-i+1}(X_1,\ldots,X_n) + u^{j-1} w h_{j-i}(X_1,\ldots,X_n) )_{1 \le i, j \le n}
\end{multline*} 
simplifies to 
\begin{multline*} 
\sum_{k=1}^n (-1)^{k-1} e_{k-1}(X_{n-i+2},\ldots,X_n) \left[ u^j h_{j-k+1}(X_1,\ldots,X_n) + u^{j-1} w h_{j-k}(X_1,\ldots,X_n) \right] \\
= u^j \sum_{k \ge 0} \langle t^k \rangle \left[ \prod_{l=n-i+2}^{n} (1-t X_l) \right]  
\langle t^{j-k} \rangle \left[ \prod_{l=1}^{n} (1-t X_l)^{-1} \right] \\
+  w u^{j-1} \sum_{k \ge 0} \langle t^k \rangle \left[ \prod_{l=n-i+2}^{n} (1-t X_l) \right]  
\langle t^{j-k-1} \rangle \left[ \prod_{l=1}^{n} (1-t X_l)^{-1} \right] \\
= u^j  \langle t^j \rangle\left[ \prod_{l=1}^{n-i+1} (1-t X_l)^{-1} \right] 
+  w u^{j-1}  \langle t^{j-1} \rangle \left[ \prod_{l=1}^{n-i+1} (1-t X_l)^{-1} \right] \\
= u^j h_j(X_1,\ldots,X_{n-i+1}) + w u^{j-1} h_{j-1}(X_1,\ldots,X_{n-i+1}) 
\end{multline*} 
Then, the entry $(i,j)$ of 
$$
\left(  e_{i-1}(X_1,\ldots,X_{n-j}) X_{n-j+1} \right)_{1 \le i,j \le n} \cdot ( u^j h_j(X_1,\ldots,X_{n-i+1}) + w u^{j-1} h_{j-1}(X_1,\ldots,X_{n-i+1}) )_{1 \le i, j \le n} 
$$
is equal to 
\begin{multline*}
 \sum_{l=1}^n e_{i-1}(X_1,\ldots,X_{n-l}) X_{n-l+1}  \left[ u^j  h_{j} (X_1,\ldots,X_{n-l+1}) 
+ w u^{j-1}  h_{j-1} (X_1,\ldots,X_{n-l+1}) \right].
\end{multline*} 
Finally, the entry $(i,j)$ of $B_n^{-1} A_n$, i.e.,
\begin{multline*} 
\left( \binom{j-1}{i-1} v^{-j} w^{j-i} \right)_{1 \le i, j \le n} \\
\cdot \left(  \sum_{l=1}^n e_{i-1}(X_1,\ldots,X_{n-l}) X_{n-l+1}  \left[ u^j  h_{j} (X_1,\ldots,X_{n-l+1}) 
+ w u^{j-1}  h_{j-1} (X_1,\ldots,X_{n-l+1})  \right] \right)_{1 \le i,j \le n}
\end{multline*}  
is equal to 
\begin{multline} 
\label{pathlast} 
\sum_{l=1}^n u^{j-1} [u h_{j} (X_1,\ldots,X_{n-l+1}) +
w h_{j-1} (X_1,\ldots,X_{n-l+1})] \\
\times \sum_{k=i-1}^n \binom{k-1}{i-1} v^{-k} w^{k-i} e_{k-1}(X_1,\ldots,X_{n-l}) X_{n-l+1}.
\end{multline} 

\subsection{Combinatorial interpretation of the determinant using the Lindstr\"om-Gessel-Viennot lemma}
\label{GVlast} 
We interpret \eqref{pathlast} as the generating function of weighted lattice paths that start in $(-i,i-2)$ and end in $(j,0)$ as follows. In fact, we will now establish the connection to the extended DPP-paths as defined in Section~\ref{dpp}.
\begin{itemize} 
\item Recall that strictly left of the $y$-axis, we have unit up-steps and unit right-steps in DPP-paths. The term 
\begin{equation}
\label{interpret}  
\sum_{k=i-1}^n \binom{k-1}{i-1} v^{-k} w^{k-i} e_{k-1}(X_1,\ldots,X_{n-l}) X_{n-l+1}
\end{equation} 
takes into account for precisely this, namely for the portion of the path from $(-i,i-2)$ to $(0,n-l-1)$, where the step that reaches $(0,n-l-1)$ is a right-step. By definition of DPP-paths the $m$-th step contributes $X_m v^{-1}$ if it is a right-step and either $1$ or $X_m v^{-1} w$ if it is an up-step. Since the last step is prescribed and has weight $X_{n-l+1} v^{-1}$, we need to choose a path from $(-i,i-2)$ to $(-1,n-l-1)$, which has in total $n-l$ steps of which $i-1$ have to be right steps. We call the $m$-th step \emph{special} if $X_m$ appears as a factor in its weight. We choose $k-1$ steps to be special, and note that this is reflected through $e_{k-1}(X_1,\ldots,X_{n-l})$ in \eqref{interpret}. Among the $k-1$ special steps we choose the positions of the $i-1$ right-steps (all right-steps are special): this is reflected in $\binom{k-1}{i-1}$ in \eqref{interpret}. The $i-1$ right-steps have each an additional multiplicative weight of $v^{-1}$ and each of the $k-i$ special up-steps  has an additional  multiplicative weight of $v^{-1} w$.
\item Next we have to deal with the portion that goes from $(0,n-l-1)$ to $(j,-2)$. This will be obtained by interpreting the term
$$
u^{j-1} [u h_{j} (X_1,\ldots,X_{n-l+1}) +
w h_{j-1} (X_1,\ldots,X_{n-l+1})]
$$
for $1 \le l \le n$.
\begin{itemize} 
\item Observe that  
$$
u^{j} h_{j} (X_1,\ldots,X_{n-l+1})
$$
actually corresponds to the paths from $(0,n-l-1)$ to $(j,-1)$ with only down-steps and right-steps.  Right-steps at height $m-2$ over the $x$-axis contribute $u X_m$ to the weight. We add to such a path a down-step at the end to reach $(j,-2)$. 
\item On the other hand, 
$$
u^{j-1}  w h_{j-1} (X_1,\ldots,X_{n-l+1})
$$
is obviously the generating function for paths from $(0,n-l-1)$ to $(j-1,-1)$, with the same weight, except that we have to multiply $w$, overall. We add to such a path a diagonal-step $(1,-1)$ to reach $(j,-2)$, and interpret $w$ as the weight of the diagonal-step.
\end{itemize} 
In summary, one can also say that right of and on the $y$-axis, we consider paths from $(0,n-l-1)$ to $(j,-2)$, with down-steps and right-step above and on the line $y=-1$, as well as
down-steps and diagonal-steps $(1,-1)$ from ``level'' $-1$ to ``level'' $-2$. Each horizontal step at level $-1$ and higher gets a weight $u X_m$, with 
$m$ is obtained by adding $2$ to the level. A diagonal-step contribute $w$ to the weight. 
\item In the end, we need to multiply with $v^{\binom{n+1}{2}}$ (recall that the weight of a family of paths is the product of the weights of the single paths times $v^{\binom{n+1}{2}}$), which comes from $\det(B_n)$. 
\end{itemize} 

Now we use the Lindstr\"om-Gessel-Viennot lemma \cite{Lin73,GesVie85,GesVie89} to interpret the determinant
\begin{equation} 
\det_{1 \le i, j \le n} \left( [i=j] + \W((-i,i-2) \to (j,-2)) \right) 
\end{equation} 
as the generating function of extended $n$-DPP paths, where $\W((-i,i-2) \to (j,-2))$ is the generating function of lattice paths from 
$(-i,i-2)$ to $(j,-2)$ with the weights as given above. Indeed, we mimic how one can be led to Andrews' definition of descending plane partitions: we can write the determinant also as 
\begin{equation} 
\sum_{r=0}^{n} \sum_{1 \le u_1 < u_2 < \ldots < u_r \le n} \det_{1 \le i, j \le r} \left( \W((-u_i,u_i-2) \to (u_j,-2)) \right). 
\end{equation}
The Lindstr\"om-Gessel-Viennot lemma now asserts that $\det_{1 \le i, j \le r} \left( \W((-u_i,u_i-2) \to (u_j,-2)) \right)$ is the weighted count of 
non-intersecting lattice paths from $A'_{u_1},\ldots,A'_{u_r}$ to $E'_{u_1},\ldots,E'_{u_r}$. This concludes the proof of Theorem~\ref{main2}.

\section{Acknowledgment} 

Ilse Fischer thanks Matja{\v z} Konvalinka for useful discussions.

\bibliographystyle{alpha}

\end{document}